\documentclass[final,1p,times]{elsarticle}
\usepackage{lineno,hyperref}
\usepackage{amsmath,amsfonts}
\usepackage{mathrsfs}
\usepackage{verbatim}
\usepackage{latexsym}
\usepackage{graphicx}
\usepackage{epstopdf}
\usepackage{subfig}
\usepackage{caption}
\usepackage{dsfont}
\usepackage{indentfirst}
\usepackage{xcolor}
\usepackage[notref,notcite]{showkeys}
\usepackage{algorithm}
\usepackage{algpseudocode}
\usepackage{booktabs}
\usepackage{natbib}
\usepackage{color}
\usepackage{amsthm}  
\allowdisplaybreaks[4]  
\newtheorem{theorem}{Theorem}[section]
\newtheorem{lemma}[theorem]{Lemma}

\newtheorem{assumption}[theorem]{Assumption}
\newtheorem{example}[theorem]{Example}

\newtheorem{remark}[theorem]{Remark}

\newcommand{\lf}{\lfloor}
\newcommand{\rf}{\rfloor}
\def\lan{\langle} \def\ran{\rangle}
\def\d{\delta}
\newcommand{\D}{\Delta}

\def\eps{\varepsilon}
\def\e{\epsilon}
\def\th{\theta}
\def\ve{\vee}
\def\we{\wedge}
\def\a{\alpha} \def\g{\gamma}

\def\f{\forall}

\newcommand{\Ito}{It\^{o} }
\newcommand{\Holder}{H\"{o}lder }
\newcommand{\BDG}{Burkholder-Davis-Gundy }

\newcommand{\E}{\mathbb{E}}

\newcommand{\PP}{\mathbb{P}}
\newcommand{\II}{\mathbb{I}} %

\newcommand{\R}{\mathbb{R}}

\newcommand{\tkk}{(k+1)\Delta}
\newcommand{\tk}{k\Delta}
\newcommand{\qu}{\quad}
\newcommand{\no}{\nonumber }
\newcommand{\intt}{\int_{-\tau}^{0}}
\newcommand{\pd}{\pi_{\Delta}}

\makeatletter
\@addtoreset{equation}{section}
\makeatother

\modulolinenumbers[5]

\begin{document}

\begin{frontmatter}
\title{
%
Segment convergence
for super-linear stochastic functional differential equations by the truncated Euler-Maruyama method
}


\author[mainaddress]{Shounian Deng}
\author[mainaddress]{Weiyin Fei}
\author[address2]{Banban Shi\corref{correspondingauthor}}
\cortext[correspondingauthor]{Corresponding author}
\ead{shibanban@zuel.edu.cn}

\address[mainaddress]{Key Laboratory of Advanced Perception and Intelligent Control of High-end Equipment, Ministry of Education,  and School of Mathematics-Physics and Finance, Anhui Polytechnic University,  Wuhu  241000, China}
\address[address2]{School of Statistics and Mathematics, Zhongnan University of Economics and Law, Wuhan, 430073, China}
\begin{abstract}
Most existing literature focuses on pointwise convergence (i.e.,  convergence at a fixed time point) of numerical solutions for Stochastic functional differential equations (SFDEs). 
In contrast, this paper investigates the strong segment convergence (i.e., the strong order of convergence of the numerical segment process).
For SFDEs with super-linear drift and diffusion coefficients, we employ the explicit truncated Euler-Maruyama (EM) scheme. First, we establish the uniform moment boundedness of the truncated EM solution over a finite time interval.
Second, we derive the 
$L^2$-error estimate between the continuous numerical segment and the step numerical segment. Finally, we prove the strong convergence order of the numerical segment generated by the truncated EM. The results can be used to analyze invariant measures and ergodicity of numerical segment, and have important applications in practical problems such as path-dependent financial options. We also provide a numerical example to support the theoretical results.

\end{abstract}


\end{frontmatter}
\nolinenumbers
\section{Introduction}
Stochastic functional differential equations (SFDEs) are used to model dynamic systems in science and industry, including physical sciences,  population dynamics, and control engineering; see  \cite{WFK2010,LuoMao2011Auto,BJH2014na,Wu2023jde,SBB2022siam}.
Since explicit solutions are generally unavailable for  SFDEs, numerical approaches are required. 
This paper is concerned with numerical methods for the  following  SFDE
\begin{equation}\label{Eq0}
\begin{cases}
dX(t)  = f(X_t)dt + g(X_t)dB(t), &  t >0,  \\[1mm]
X(t)  = \xi (t), &  t \in [-\tau, 0],
\end{cases}
\end{equation}
where $\tau >0$ is a constant, $f: C([-\tau, 0 ]; \R^{n_1}) \to \R^{n_1}$ and $g: C([-\tau, 0 ]; \R^{n_1})\to \R^{n_1 \times n_2}$ are local Lipschitz continuous, $B(\cdot)$ is an $n_2$-dimensional Brownian motion
on a filtered probability space
 $(\Omega , {\mathcal F}, \{{\cal F}_t\}_{t\ge 0}, \mathbb{P})$
satisfying the usual conditions (i.e., it is increasing and right continuous while $\cal{F}_\textrm{0} $ contains all $\mathbb{P}$\textrm{-}null sets), the initial data $\xi (\cdot)\in
C([-\tau, 0 ]; \R^{n_1}) $, and segment process $ X_t = \{ X_t (\th): \th \in [-\tau, 0] \}$ is a $C([-\tau, 0 ]; \R^{n_1}) $-valued stochastic process for $t \ge0$, in which  $X_t (\th) = X(t+ \th)$ with  $\| X_t \|: = \sup_{-\tau \le \th \le 0} | X(t+ \th)|$. 
\par 

With respect to numerical solutions for SFDE \eqref{Eq0}, so far most of the literature has addressed pointwise convergence \cite{Mao2003LMS,Zhao2020ANM,Li_Mao2024spa}, for example, the mean-square convergence at a fixed point
\begin{align}\label{eq2026421_2}
\lim_{ \D \to 0} \E |X(T) - Y^{\D} (T)|^2 =0, 
\end{align}
where $\D$ denotes the step size,  $X(T)$ and $Y^{\D}(T)$ denote the exact and the corresponding numerical solutions, respectively.
On the contrary, this paper focuses on the strong order of the numerical  segment process $Y^{\D}_T$, i.e., 
\begin{align}\label{eq2026421}
 \E\|X_T -  Y^{\D}_T \|^{2} 
\le C\D^{\hat \g },
\end{align}
where $C$ and $\hat \g$ are positive constants, $X_t$ and $Y^{\D}_t$ denote the exact and the corresponding numerical segments, respectively.
\par
It is clear that the numerical segment convergence \eqref{eq2026421} has its own importance. Such convergence may help to derive the ergodicity of numerical solutions, similar to that of the underlying equations. For example, suppose that the SFDE admits a unique invariant probability measure. Then, with the help of numerical segment convergence \eqref{eq2026421}, we can conclude that the discrete-time semigroup associated with the numerical segment also admits a unique numerical invariant probability measure under uniform moment boundedness, attractivity, and Markov property of the numerical segment \cite{Shi2024jde}. Moreover, in some financial situations, we need to approximate path-dependent quantities, such as the European barrier option value \cite{Higham2005}. In these situations, we require the strong convergence of a numerical segment.

For SFDEs with linearly growing coefficients, several numerical methods have been well established, including the Euler-Maruyama (EM) scheme \cite{Mao2003LMS,Zhou2025camc}, the theta-EM method \cite{2004The}, and the split-step backward EM method \cite{Jiang2011cns}. While implicit approaches such as the backward EM \cite{Zhou2017CAM,Shi2024jde} and split-step theta methods \cite{Zhao2020ANM} remain effective for SFDEs with super-linear drift and linear diffusion,  their computational efficiency is constrained by the requirement to solve a nonlinear equation at each time step.

Recently, Li, Mao and Song \cite{Li_Mao2024spa} introduced an explicit truncated EM scheme for SFDEs with super-linear  coefficients, 
removing the conventional linear growth condition on the diffusion term required in 
earlier works \cite{Zhou2017CAM,Zhang_song2018CAM}.
They established a strong convergence rate of order $1/2$ for the pointwise  process $X(t)$ in  SFDE \eqref{Eq0}. 
The truncated EM can also be applied to approximate the solution of SFDEs with infinite time delay \cite{Zhou_Mao2026}.

Based on the above discussions, the objective of this work is to study the strong order of the numerical segment  for SFDEs with super-linear drift and diffusion coefficients.
In order to
handle the super-linear growth of the coefficients,
we employ the truncated methods from \cite{Mao2015truncated,Li_Mao2024spa}. 
We first establish 
the uniform moment boundedness of the truncated EM solution $Y(t)$ over a finite time interval $[0,T]$.
Then, we provide the $L^2$-errors between the continuous and the step numerical segments. Finally, we prove the strong convergence order of the  numerical segment generated by the truncated EM.

The rest of the paper is organized as follows. 
Section 2 introduces some notations and preliminary results on the exact solutions of SFDE \eqref{Eq0}. Section 3  proposes an truncated EM scheme for SFDE \eqref{Eq0}. In Section 4, the  strong convergence order of segment sequence 
generated by the truncated EM
 is established on a finite time interval.  In Section 5, a numerical example is presented to illustrate the theoretical results.
\section{Preliminary results}

Throughout this paper, the following notations are used.  Let $| \cdot |$ denote both the
Euclidean norm in $\R^{n_1}$ and the trace norm in $\R^{ n_1 \times n_2}$.
For two real numbers $a$ and $b$,  $a \ve b:  = \max (a,b)$ and $a \we b : = \min (a,b)$.
For a set $G$, its indicator function is denoted by $\II_G$.
Denote by $\mathcal{C_{\tau}}:= C([-\tau, 0 ]; \R^{n_1}) $ the space of all continuous function $\psi (\cdot) $ from $[-\tau, 0]$ to $\R^{n_1}$ equipped with the norm $\| \psi \| = \sup_{-\tau \le \th \le 0} | \psi (\th)|$.
Let $\mathcal{B}(\mathcal{C_\tau})$ be the Borel $\sigma$-algebra in $\mathcal{C_\tau}$ and $\mathcal{B}_{b}(\mathcal{C_{\tau}})$ be the set of all bounded measurable functions on $\mathcal{C_{\tau}}$. 
$\mathbf 0$ denotes the zero element in $\mathcal{C_{\tau}}$, namely, $ \mathbf 0 (\th) =0$, for any $\th \in [-\tau, 0]$.
For any $x, y\in \mathbb{R}^{n_1}$, their inner product is denoted by $\lan x,y\ran$.
Let $\R_+ = [0, \infty)$. We denote by $\mathbb N_+$ the set of strictly positive integers.
\par
We now introduce  the following assumptions.

\begin{assumption}\label{AS2_Kha}
There exist positive constants $q > 3$, $\hat r $, $\a_0$, $\a_1 >\a_2  $ such that
\begin{align*}
& 2 \lan \psi (0)- \bar \psi (0), f(\psi) - f(\bar \psi) \ran + (q -1) |g(\psi) - g(\bar \psi)|^2 \le \a_0 \Big ( | \psi (0) - \bar \psi (0)|^2+ 
\frac{1}{\tau} \int_{-\tau}^{0} | \psi (\th) - \bar \psi (\th)|^2 
  d \th
\Big ) \\
& - \a_1  | \psi (0) - \bar \psi (0)|^2 
 (|\psi (0)|^{\hat r}+
  |\bar \psi (0)|^{\hat r})
  +  \frac{\a_2}{\tau} \int_{-\tau}^{0} | \psi (\th) - \bar \psi (\th)|^2 
 (|\psi (\th)|^{\hat r} + |\bar \psi (\th)|^{\hat r})
  d \th,
\end{align*}
for any $\psi, \bar \psi \in \mathcal{C_{\tau}}$.
\end{assumption}
\begin{assumption}\label{AS5_POLYNO}
There exist positive constants $c_1$, $r$ such that
\begin{align*}
 & |f(\psi) - f(\bar \psi)|  
  \le c_1 \Big (| \psi (0) - \bar \psi (0)|(1+  | \psi (0)|^{r} +  | \bar \psi (0)|^{r} )
 + \frac{1}{\tau} \int_{-\tau}^{0} | \psi (\th) - \bar \psi (\th)| 
(1+  | \psi (\th)|^{r} +  | \bar \psi (\th)|^{r} ) d \th \Big ), \\
& |g(\psi) - g(\bar \psi)|^2  
 \le c_1 \Big (| \psi (0) - \bar \psi (0)|^2(1+  | \psi (0)|^{r} +  | \bar \psi (0)|^{r} ) 
 + \frac{1}{\tau} \int_{-\tau}^{0} | \psi (\th) - \bar \psi (\th)|^2 
(1+  | \psi (\th)|^{r} +  | \bar \psi (\th)|^{r} ) d \th \Big ),
\end{align*}
for any $\psi, \bar \psi \in \mathcal{C_{\tau}}$.
\end{assumption}
\begin{assumption}\label{AS3}
There exist positive constant $c_2$  such that the initial data $\xi$ satisfies
\begin{align*}
|\xi (t_1) - \xi (t_2)|^2 \le c_2 |t_1 - t_2|, \quad \f t_1, t_2 \in [-\tau, 0].
\end{align*}
\end{assumption}
\begin{remark}\label{rem407}
In many models used in mathematical finance, Assumptions  \ref{AS2_Kha} and \ref{AS5_POLYNO} are met. Typical example is the functional version of  the stochastic volatility model \cite{FSSCA2017,HS2015ams}
\begin{equation}\label{example407}
\begin{cases}
dX(t) = \left(a_{0} + a_{1} X(t) - a_{2} X^3(t)  \right) dt 
+  \int_{-1}^{0}X^2(t+\theta)d \theta d B(t), &  t >0,  \\[1mm]
X(t)  = 0.05, &  t \in [-1, 0],
\end{cases}
\end{equation}
%
 where $a_0$, $a_1$ and $a_2$ are positive constants. See the Example \ref{example4072} for  verification of the assumptions.
\end{remark}

\textbf{In what follows, we take the constant $p=q - \hat \epsilon$,  where $q$  is from Assumption \ref{AS2_Kha} and $\hat \epsilon$ is a sufficiently small positive constant.} Moreover, 
$C$ denotes a generic positive constant, whose value may change from line to line. Similarly, denote by $C(\alpha)$, $C(\alpha, \beta)$ the generic positive constants depending on parameter $\alpha$ and $(\alpha, \beta)$, respectively.\par
By the elementary $(a + b)^2 \le (1 + \hat \e )a^2 + (1 + 1/\hat \e)b^2 $, for $a, b \in \R$, $\hat \e >0$, we have 
\begin{align*}
 & 2 \lan \psi (0), f(\psi ) \ran + (p-1) |g(\psi)|^2  \\
& = 2 \lan \psi (0) -\mathbf 0  (0) , f(\psi ) -f(\mathbf 0 ) \ran + 
2 \lan \psi (0) -\mathbf 0  (0) , f(\mathbf 0 ) \ran + (p-1) |g(\psi) - g(\mathbf 0 ) + g(\mathbf 0 )|^2 \\
& \le  2 \lan \psi (0) -\mathbf 0  (0) , f(\psi ) -f(\mathbf 0 ) \ran + 
(p-1)(1 + \hat \e) |g(\psi) - g(\mathbf 0 )|^2  \\
& \qu + 2 \lan \psi (0)  , f(\mathbf 0 ) \ran + (p-1)(1 + 1/ \hat \e) | g(\mathbf 0 )|^2 .
\end{align*}
Letting $(p-1)(1 + \hat \e) = q -1$, we conclude from Assumption \ref{AS2_Kha} that 
\begin{align}\label{AS4}
 & 2 \lan \psi (0), f(\psi ) \ran + (p-1) |g(\psi)|^2  \\
 & \le c_3 \Big ( 1 + |\psi (0)|^2 + \frac{1}{\tau} \int_{-\tau}^{0} | \psi (\th)|^2 d \th  \Big )
 - \a_1 | \psi(0)|^{2+  \hat r} +
   \frac{\a_2}{\tau} \int_{-\tau}^{0} | \psi (\th)|^{2+  \hat r}   d \th , \qu \f \psi \in \mathcal{C_{\tau}} , \no
\end{align}
where $c_3$ is a positive constant. 
From Assumption \ref{AS5_POLYNO}, it follows that for any $\psi \in \mathcal{C_{\tau}}$,
\begin{align}
& |f(\psi)| \le C \Big ( 1 + |\psi (0)|^{1+ r} + \frac{1}{\tau}\intt | \psi(\th) |^{1 +r} d \th \Big ), 
\label{Eq_B11}  \\
& |g(\psi)|^2 \le C \Big ( 1 + |\psi (0)|^{2+ r} + \frac{1}{\tau}\intt | \psi(\th) |^{2 +r} d \th \Big ). \label{Eq_B1} 
\end{align}
It is noted that Assumption \ref{AS5_POLYNO} implies the local Lipschitz condition.
According to Theorem [2.1] of \cite{Li_Mao2024spa}, under Assumptions \ref{AS2_Kha} and \ref{AS5_POLYNO}, the SFDE \eqref{Eq0} has a unique solution $X(t)$.
For any $R > \| \xi \|$, 
define
\begin{align}\label{Eq9434}
\tau_R:= \inf \{ t \ge 0: |X(t)| \ge R \}.
\end{align}
In a similar argument as in the derivation of   \cite[p. 1231]{Wu2017jde}, we can show that
$\tau_R= \inf \{ t \ge 0: \|X_t \| \ge R \}$.
\begin{lemma}\label{lem2.11}
Let  Assumptions \ref{AS2_Kha} and \ref{AS5_POLYNO}  hold. 
Then SFDE \eqref{Eq0} with the initial data $\xi \in \mathcal{C_{\tau}}$ has a unique global solution $X(\cdot)$ on $[-\tau, \infty )$ satisfying \begin{align}\label{Bound_Xp}
\sup_{ -\tau \le t < \infty} \E |X(t)|^p \le C(p) (1 + \|\xi \|^p) .
\end{align}
For any $T >0$,  it holds that
\begin{align}\label{Eq95}
\PP \{ \tau_R \le T \} \le \frac{C(p,T)}{R^p},
\end{align}
where $ \tau_R$ is defined by \eqref{Eq9434},   $C_T$ depends on $T$ but is independent of $R$. Moreover, 
\begin{align}\label{Bound_XX3}
 \E \left [ \sup_{ 0 \le t \le T } |X(t)|^{\hat p}  \right ]\le C(\hat p, T) 
%
  (1 + \|\xi \|^{\hat p + ( r \ve \hat r) }),
 \qu \textrm{with} \qu  2 \le  \hat p  \le  p -r.
\end{align}
\end{lemma}
\begin{proof}
  The proofs of \eqref{Bound_Xp} and \eqref{Eq95} can be found in \cite[Theorem 2.1]{Li_Mao2024spa}.
We now begin to prove \eqref{Bound_XX3}.
Let $2 \le \hat    p  \le  p-r$. Applying the Young inequality, we conclude from  \eqref{AS4}  that
\begin{align}\label{SS2}
 & |\psi (0)|^{\hat p -2} \Big (  2 \psi (0) ^T f(\psi ) + (\hat p-1) |g(\psi)|^2  \Big ) \no \\
 & \le C(\hat p)  \Big ( |\psi (0)|^{\hat p -2} + |\psi (0)|^{\hat p} + 
 \frac{1}{\tau} \int_{-\tau}^{0} |\psi (0)|^{\hat p -2}| \psi (\th)|^2 d \th  \Big ) \no \\
 & \qu - \a_1 | \psi(0)|^{\hat p+  \hat r} +
   \frac{\a_2}{\tau} \int_{-\tau}^{0}|\psi (0)|^{\hat p -2}  | \psi (\th)|^{2+  \hat r}   d \th  \no \\
 & \le   C(\hat p)  
 \Big (
 1 + |\psi (0)|^{\hat p} + 
 \frac{1}{\tau} \int_{-\tau}^{0} \Big ( \frac{\hat p -2}{\hat p} |\psi (0)|^{\hat p } +  \frac{2}{\hat p}| \psi (\th)|^{\hat p} \Big  ) d \th 
  \Big ) \no \\ 
   & \qu - \a_1 | \psi(0)|^{\hat p+  \hat r} +
   \frac{\a_2}{\tau} \int_{-\tau}^{0}
   \Big ( \frac{\hat p-2}{\hat p+\hat r} |\psi (0)|^{\hat p + \hat r}  +  \frac{2+ \hat r}{\hat p+\hat r}| \psi (\th)|^{\hat p+  \hat r}   \Big )d \th  \no \\
   & 
   \le   C(\hat p)  
 \Big (
 1 + |\psi (0)|^{\hat p} + 
 \frac{1}{\tau} \int_{-\tau}^{0}| \psi (\th)|^{\hat p}  d \th   \Big ) \no \\
 & \qu - \Big ( \a_1 -   \frac{\hat p-2}{\hat p+\hat r}  \a_2 \Big ) | \psi(0)|^{\hat p+  \hat r}
 + \a_2\Big (  \frac{2+ \hat r}{\hat p+\hat r}\Big ) \frac{1}{ \tau} 
 \int_{-\tau}^{0}  | \psi (\th)|^{\hat p+  \hat r}  d \th, \qu \f \psi \in \mathcal{C_{\tau}}. 
\end{align}
Using the \Ito formula and \eqref{SS2} yields
\begin{align}\label{SS4}
& |X(t)|^{\hat p} -  |X(0)|^{\hat p}\no  \\
&  \le   \frac{\hat p}{2} \int_{0}^{t} |X(s)|^{\hat p-2}
\Big ( 2 X^T (s) f(X_s)+ (\hat p - 2) |g(X_s)|^2 \Big ) ds + 
\underbrace{ \hat p \int_{0}^{t} |X(s)|^{\hat p-2}X^T (s) g(X_s)dB(s)}_{:= M(t)}   \no \\
&\le C(\hat p)\frac{\hat p}{2} \int_{0}^{t}\Big ( 1+ |X(s)|^{\hat p} + \frac{1}{\tau}\intt |X(s+\th)|^{\hat p} d \th \Big )ds  + M(t) \no \\
& \qu - \frac{\hat p}{2} \left ( \Big ( \a_1 -   \frac{\hat p-2}{\hat p+\hat r}  \a_2 \Big ) 
\int_{0}^{t}|X(s)|^{\hat p+ \hat r} ds 
+ \a_2\Big (  \frac{2+ \hat r}{\hat p+\hat r}\Big )   \frac{1}{\tau}
\int_{0}^{t}\Big (\intt |X(s+\th)|^{\hat p + \hat r} d \th \Big )ds  \right ), \qu t \ge 0 . 
\end{align}
According to 
 the Fubini Theorem, we have
\begin{align*}
& \frac{1}{\tau} \int_{0}^{t} \Big ( \intt |X(s+ \th)|^{\hat p + \hat r}   d\th \Big )ds
= \frac{1}{\tau}  \intt  \Big ( \int_{0}^{t} |X(s+ \th)|^{\hat p + \hat r} ds \Big )   d\th \\
& \le \frac{1}{\tau} \intt  \Big ( \int_{-\tau}^{t}  |X(s)|^{\hat p + \hat r}  ds \Big ) d\th
=\int_{-\tau}^{t}  |X(s)|^{\hat p + \hat r}  ds
 \le \int_{0}^{t}  |X(s)|^{\hat p + \hat r}  ds + \tau  \| \xi \|^{\hat p + \hat r}
.  
\end{align*}
Inserting this \eqref{SS4} gives that 
\begin{align}\label{SS5}
 |X(t)|^{\hat p}
&\le C(\hat p) \int_{0}^{t}\Big ( 1+ \| \xi \|^{\hat p }+ |X(s)|^{\hat p}  \Big )ds  + M(t) \no \\
& \qu - \frac{\hat p}{2} \Big ( \a_1 -   \frac{\hat p-2}{\hat p+\hat r}  \a_2   -  \frac{2+ \hat r}{\hat p+\hat r} \a_2 \Big ) 
\int_{0}^{t}|X(s)|^{\hat p+ \hat r} ds 
+ \a_2\Big (  \frac{2+ \hat r}{\hat p+\hat r}\Big ) \tau  \| \xi \|^{\hat p + \hat r} \frac{\hat p}{2}.
\end{align}
Noting that $\displaystyle \a_1 -   \frac{\hat p-2}{\hat p+\hat r}  \a_2   -  \frac{2+ \hat r}{\hat p+\hat r} \a_2  = \a_1 - \a_2 >0$, we conclude from \eqref{SS5} that 
\begin{align}\label{SS6}
 |X(t)|^{\hat p}
&\le C(\hat p) \int_{0}^{t}\Big ( 1+ \| \xi \|^{\hat p + \hat r }+ |X(s)|^{\hat p}  \Big )ds  + \hat p \int_{0}^{t} |X(s)|^{\hat p-2}X^T (s) g(X_s)dB(s), \qu t \ge 0.
\end{align}
By  the \BDG inequality,
 we then derive that 
\begin{align}\label{SS88}
\E \left [ \sup_{0 \le t \le T} |X(t)|^{\hat p} \right ]  & \le 
C(\hat p) (1+ \| \xi \|^{\hat p + \hat r }) + \hat p
\E \left [ \sup_{0 \le t \le T} \Big |  \int_{0}^{t} |X(s)|^{\hat p-2}X^T (s) g(X_s)dB(s)  \Big | \right ] \no \\
& \le C(\hat p) (1+ \| \xi \|^{\hat p + \hat r }) 
+ 4 \sqrt{2} \hat p \E  \left [ \Big (\int_{0}^{T} |X(s)|^{2 \hat p -2} |g(X_s)|^2 ds \Big )^{1/2}   \right ] \no \\
&\le  C(\hat p) (1+ \| \xi \|^{\hat p + \hat r }) + 
4 \sqrt{2} \hat p \E  \left [\Big ( \sup_{0 \le s \le T} |X(s)|^{\hat p} \Big ) ^{1/2} 
 \Big (\int_{0}^{T} |X(s)|^{ \hat p -2} |g(X_s)|^2 ds \Big )^{1/2}   \right ] \no \\
&\le   C(\hat p) (1+ \| \xi \|^{\hat p + \hat r }) + 
\frac{1}{2}\E \left [ \sup_{0 \le s \le T} |X(s)|^{\hat p} \right ] 
+ 
16 \hat p^2 \E \left [  \int_{0}^{T} |X(s)|^{ \hat p -2} |g(X_s)|^2 ds \right ]. 
\end{align}
By  \eqref{Eq_B1} and the Young inequality as well as  \eqref{Bound_Xp}, we derive that
\begin{align}\label{SS77}
\E \left [  \int_{0}^{T} |X(s)|^{ \hat p -2} |g(X_s)|^2 ds \right ] 
& \le 
C \E \left [  \int_{0}^{T} |X(s)|^{ \hat p -2} 
\Big ( 1 + |X(s)|^{2+ r}  + \frac{1}{\tau}\intt |X(s+\th)|^{2+r} d \th  \Big )
 ds \right ]  \no \\
& \le C(\hat p, T) 
\E \left [  \int_{0}^{T}
\Big ( 1 + |X(s)|^{\hat p + r}  + \frac{1}{\tau}\intt |X(s+\th)|^{\hat p +r} d \th  \Big )
 ds \right ] \no \\
 &\le  C(\hat p, T) (1 + \| \xi \|^{\hat p +r})
\E \left [  \int_{0}^{T}
\Big ( 1 + |X(s)|^{\hat p + r}  \Big ) 
 ds \right ]  \no \\
& \le  C(\hat p, T) (1 + \| \xi \|^{\hat p +r}), 
\end{align}
where we have used the following estimate
\begin{align*}
& \frac{1}{\tau} \int_{0}^{T} \Big ( \intt |X(s+ \th)|^{\hat p +  r}   d\th \Big )ds
  \le \int_{-\tau}^{T}  |X(s)|^{\hat p +  r}  ds 
 \le \int_{0}^{T}  |X(s)|^{\hat p +  r}  ds + \tau  \| \xi \|^{\hat p +  r}
.  
\end{align*}
Inserting \eqref{SS77} into \eqref{SS88}, we get desired assertion \eqref{Bound_XX3}. Thus, the proof is finished. 
\end{proof}  
\section{Truncated EM method}
In this section, we construct an explicit scheme for SFDE \eqref{Eq0}. Under Assumption \ref{AS5_POLYNO},
we first choose a strictly increasing continuous function  $H: [1, \infty) \to \R_+$ such that $H (R)\to \infty$ as $R\to \infty$ and
\begin{align}\label{eq21}
 \sup_{\| \psi \| \ve \| \bar \psi \| \le R }
 \frac{|f(\psi) - f(\bar \psi )|}{(\Psi(\psi, \bar \psi )^{1/2}}
 \ve  \frac{|g(\psi) - g(\bar \psi )|^2}{\Psi(\psi, \bar \psi )}
\le H (R) , \quad  \f  R \ge 1,
\end{align}
where $\psi, \bar \psi \in \mathcal{C_{\tau}}$,  $\psi \neq \bar \psi$, and
\begin{align*}
\Psi(\psi, \bar \psi ):=
 | \psi (0) - \bar \psi (0)|^2 + \frac{1}{\tau} \int_{-\tau}^{0} | \psi (\th) - \bar \psi (\th)|^2  d \th.
\end{align*}
Denote $H^{-1}$ by the inverse function of $H$.
For a step size $\D \in (0,1]$ which is a fraction of $\tau$, say $\D = \tau / m$ for some integer $m \ge \tau$,
let us define a truncation mapping $\pd: \R^{n_1} \to \{ x \in \R^{n_1}: |x| \le H^{-1}(c_4 \D^{-\varrho} ) \}$ by
\begin{align}\label{eq23}
\pd(x) = \left ( |x| \we H^{-1}(c_4\D^{-\varrho} )  \right )\frac{x}{|x|}, \; \f x \in \R^{n_1},
\end{align}
where $ \varrho \in (0, 1/2)$ is a constant to be determined, $c_4: =  H(\|\xi \|) \ve H(1) \ve |f(\mathbf 0)|\ve |g(\mathbf 0)|$. Define $x/|x| = 0$ when $x =0$.
Define the truncated EM scheme for SFDE \eqref{Eq0} by
\begin{equation}\label{TEM1}
\begin{cases}
  \hat Y(\tk)=\xi(\tk),  & k = -m, -m+1, \cdots, 0,  \\
   Y(\tk)=\pd ( \hat Y(\tk)),  & k = -m, -m+1, \cdots,   \\
\hat Y(\tkk) = Y(\tk )+ f(Y_{\tk})\D + g(Y_{\tk}) \Delta B_{k\D}, & k = 0, 1,2, \cdots,
\end{cases}
\end{equation}
where $\Delta B_{k\D}=B(\tkk)-B(\tk)$, $Y_{k\Delta}=\{Y_{\tk}(\theta) : -\tau\leq\theta\leq0\}$ is a $\mathcal{C_{\tau}}$-valued random variable defined as follows: for $\theta \in [i\Delta,(i+1)\Delta]$, $i=-m, -m+1, \cdots, -1$,
\begin{align}\label{Eqq3.2}
Y_{\tk}(\theta)
& = \frac{(i+1)\Delta-\theta}{\Delta}Y((k+i)\D)
+\frac{\theta-i\Delta}{\Delta}Y((k+i+1)\D). 
\end{align}
That is, $Y_{\tk}(\cdot)$ is the linear interpolation of $Y((k-m)\D)$, $Y((k-m+1)\D)$, $\cdots$, $Y(\tk)$.
Correspondingly, we call
$Y_{\tk}$
as the numerical segment  generated by the truncated EM scheme \eqref{TEM1}.
For any $\hat p\ge2$, by the convex property of $u(x) = x^{\hat p}$, we conclude from \eqref{Eqq3.2} that
\begin{eqnarray*}
|Y_{\tk}(\theta)|^{\hat p} & \leq&\frac{(i+1)\Delta-\theta}{\Delta}|Y((k+i)\D)|^{\hat p}
+\frac{\theta-i\Delta}{\Delta}|Y((k+i+1)\D)|^{\hat p}\\
&\leq&|Y((k+i)\D)|^{\hat p} \ve |Y((k+i+1)\D)|^{\hat p}, \nonumber
\end{eqnarray*}
for $\theta \in [i\Delta,(i+1)\Delta]$, $i=-m, -m+1, \cdots, -1$.
Then
\begin{equation}\label{PP0}
\|Y_{\tk}\|^{\hat p}=\max _{-m\leq i\leq0}|Y((k+i)\D)|^{\hat p}
~~\text{and}~~
\|Y_{\tk}\|^{\hat p}\leq\|Y_{(k-1)\D}\|^{\hat p} \vee|Y(\tk)|^{\hat p}.
\end{equation}
From Assumption \ref{AS5_POLYNO} and \eqref{eq21}, we choose 
$H(R):= C R^r, \; \f R \ge 1.$
Thus,
\begin{align}\label{PP}
\| Y_{\tk} \| \le H^{-1}(c_4\D^{-\varrho})=C\D^{-\varrho/r}.
\end{align}
Then it follows from this and \eqref{eq21} that
\begin{equation}\label{PP2}
  \begin{cases}
\displaystyle |f(Y_{\tk})|^2  \le C \D^{-2\varrho} \Big ( 1+ |Y(\tk)|^2 + \frac{1}{\tau} \int_{-\tau}^{0} |Y_{\tk} (\th)|^2 d \th \Big )  \\
\displaystyle |g(Y_{\tk})|^2  \le C \D^{-\varrho} \Big (1+ |Y(\tk)|^2 + \frac{1}{\tau} \int_{-\tau}^{0} |Y_{\tk} (\th)|^2 d \th \Big ).
  \end{cases}
\end{equation}
Moreover,
\begin{align}\label{PP3}
|\pd (x) | \le |x| \qu \textrm{and} \qu |\pd (x) - \pd (y)| \le |x-y|, \qu \f x,y \in \R^{n_1},
\end{align}
see \cite[p. 880]{LiMao2019IMA}.
For any $t \in [0, \infty)$,  define an auxiliary process $Z(\cdot)$ by
\begin{equation}\label{Eq932}
  \begin{cases}
\displaystyle Z(t) = Y(\tk) + f (Y_{\tk}) (t - \tk) + g(Y_{\tk}) (B(t) - B(\tk)), &t \in [\tk, \tkk), \\
\displaystyle Z(t) = \xi (t), & t \in [-\tau, 0].
  \end{cases}
\end{equation}
For any $t \ge 0$, define
\begin{align}
&Z_t (\th) = Z(t + \th), \qu \f \th \in [-\tau, 0], \label{Geq1} \\
& \bar Y_t  = \sum_{k=0}^{\infty} Y_{\tk}\II_{[\tk, \tkk)} (t), \qu \f t \in [0, \infty). \label{Geq2} 
\end{align}
Define 
\begin{align}\label{Def1}
\underline t  = \lfloor t/\D \rfloor \D ,   \qu \f t \in [0, \infty).
\end{align}
Then we  have 
\begin{align}\label{G94}
\bar Y_t(0) = Y_{\underline t} (0) = Y(\underline t), \qu \f t \in [0, \infty).
\end{align}
For any  $R > \| \xi \|$, 
define
\begin{align}\label{Eq94322}
\rho_{R, \D}:= \inf \{ t \ge 0: |Z(t)| \ge R \}.
\end{align}
Then, $\rho_{R,\D}= \inf \{ t \ge 0: \|Z_t \| \ge R \}$.
          From \eqref{Eq932} we see that  $Z(\tk) = Y(k\D)$  and 
$ \lim_{t \uparrow \tk } Z(t) = \hat Y(\tk )$  for any $ t \in [(k-1)\D, \tk)$.
However, 
$Z(\cdot )$ is continuous in  $[-\tau, \rho_{R, \D} ]$.
 Thus, 
\begin{align}\label{Eq9444}
Z(t  \we \rho_{R,\D} ) = Y(0) + \int_{0}^{t  \we \rho_{R,\D}} f(\bar Y_s) ds + \int_{0}^{t  \we \rho_{R,\D}} g(\bar Y_s) dB(s).
\end{align}   

\section{Strong convergence order of  numerical segment}
In this section, we shall show the convergence order of the numerical segment generated by the truncated EM in the sense of  $L^2$. We now consider the uniform moment boundedness of numerical solution.     
\begin{lemma}\label{Lem3.12}
Let  Assumptions \ref{AS2_Kha} and \ref{AS5_POLYNO} hold
with $ 2 + (r\ve \hat r)  \le p <q $. Then
\begin{align}
\sup_{0 < \D \le 1}\sup_{ -m \le k < \infty} \E |Y(\tk)|^p & \le C(p) (1 + \|\xi \|^p) , \qu 
 \label{Bound_Xp312}  \\ 
\sup_{0 < \D \le 1}\sup_{ -\tau \le t < \infty} \E |Z(t)|^p & \le C(p) (1 + \|\xi \|^p) , \qu \label{Bound_Xp312_29}  \\
\sup_{0 < \D \le 1}\sup_{ -\tau \le t < \infty} \E  |\hat Y(\tk)|^p & \le C(p) (1 + \|\xi \|^p) .
\qu \label{Bound_Xp312_30} 
\end{align} 
For any $T >0$,  it holds that
\begin{align}\label{Eq952}
\PP \{\rho_{R,\D} \le T \} \le \frac{C(p,T)}{R^p},
\end{align}
let 
$ \displaystyle \hat p \in \Big  [2,   \frac{p}{1 + ( r \ve \hat r) /2} \Big ]$, 
then
\begin{align}
 \E \left [ \sup_{ -\tau \le k\D \le T }    |Y(\tk)|^{\hat p} \right ] & \le C(\hat p,T) (1 + \|\xi \|^{\hat p (1 + (r \ve \hat r)/2)}),  \label{Bound_XX242}   \\
  \E \left [ \sup_{ 0 \le t \le T }    | Z(t)|^{\hat p} \right ]  & \le C(\hat p,T) (1 + \|\xi \|^{\hat p (1 + (r \ve \hat r)/2)}),
  \label{Bound_XX2422} \\
   \E \left [ \sup_{ -\tau \le k\D \le T }    |\hat Y(\tk)|^{\hat p} \right ] & \le C(\hat p,T) (1 + \|\xi \|^{\hat p (1 + (r \ve \hat r)/2)}).  \label{Bound_XX24266}   
\end{align}
\end{lemma}
\begin{proof}
Assertions \eqref{Bound_Xp312} and \eqref{Bound_Xp312_29} follow  by \cite[Theorem 3.1]{Li_Mao2024spa} and 
\cite[Lemma 3.5]{Li_Mao2024spa}, respectively.
 The proof of \eqref{Eq952} can be found in \cite[Lemma 3.3]{Li_Mao2024spa}. \par
 By \eqref{TEM1} and \eqref{PP2}, we have
 \begin{align*}
 \E  |\hat Y(\tk)|^p 
 & \le 3^{p-1} \left (  \E  | Y(\tk)|^p + \D^p \E  | f (Y_{\tk})|^p +\D^{p/2} \E  | g (Y_{\tk})|^p \right ) \\
 & \le C(p) \left (  1+ \D^{p(1 - \varrho)}  + \D^{\frac{p(1 - \varrho)}{2}}  \right )
 \E \left ( 1 +  | Y(\tk)|^p  + \frac{1}{\tau}\intt | Y_{\tk}(\th)|^p d \th  \right ) \\ 
 & \le C(p) \Big ( 1  + \sup_{0 < \D \le 1} \sup_{k \ge -m} \E| Y(\tk)|^p  \Big ).
 \end{align*}
 Thus, the desired assertion \eqref{Bound_Xp312_30} follows from  \eqref{Bound_Xp312}. 
  
\par 
We now begin to proof \eqref{Bound_XX242}, the proof of which
can be shown by a modification of the proof of Lemma 3.7 in \cite{Higham_Mao2002}.
In light of \eqref{TEM1} and \eqref{PP3}, we arrive at
\begin{align}\label{Eq_E1}
& |Y(\tkk)|^2 \\
& \le |\hat Y(\tkk)|^2 = |Y(\tk )+ f(Y_{\tk})\D + g(Y_{\tk}) \Delta B_{k\D}|^2 \no \\
& = |Y(\tk)|^2 + 2 \lan Y(\tk), f(Y_{\tk}) \ran \D + |g(Y_{\tk}) \D B_{k\D}|^2  \no \\
& \qu + |f(Y_{\tk})|^2 \D^2 +  2\lan Y(\tk), g(Y_{\tk}) \D B_{k\D} \ran + 2 \D \lan f(Y_{\tk}, g(Y_{\tk}) \D B_{k\D} \ran \no \\
& \le |Y(\tk)|^2 +  2 \lan Y(\tk), f(Y_{\tk}) \ran \D + 2|g(Y_{\tk}) \D B_{k\D}|^2 + 2|f(Y_{\tk})|^2 \D^2 + J_3(k) \no \\
& \le |Y(\tk)|^2 +  2   (\lan Y(\tk), f(Y_{\tk}) \ran  + |g(Y_{\tk}) |^2 ) \D+ 2|f(Y_{\tk})|^2 \D^2 + J_3(k) + J_4(k). \no
\end{align}
where 
\begin{align*}
J_3(k)=2\lan Y(\tk), g(Y_{\tk}) \D B_{k\D} \ran, \qu \textrm{and}, \qu 
J_4(k)=|g(Y_{\tk})|^2 \big ( |\D B_{k\D}|^2 - \D \big ) .
\end{align*}
It follows from \eqref{AS4} and \eqref{PP2} that
\begin{align}\label{Eq_E444}
 |Y(\tkk)|^2 & \le |Y(\tk)|^2 +
C\D\Big (1+ |Y(\tk)|^2 + \frac{1}{\tau} \int_{-\tau}^{0} |Y_{\tk} (\th)|^2 d \th \Big )  \no \\
& \qu + C\D  \int_{-\tau}^{0} |Y_{\tk} (\th)|^{2+ \hat r}  d \th  + J_3(k) + J_4(k), \qu k \ge 0.
\end{align}
Assume that $n$ and $N$ are positive integers such that  $  n \le N = \lf T /\D \rf $.
Then, we conclude from \eqref{Eq_E444} that
\begin{align}\label{TT_1}
 |Y(n\D )|^2 & \le |\xi (0)|^2 +
C\D \sum_{k=0}^{n-1}  \Big (1+ |Y(\tk)|^2 + \frac{1}{\tau} \int_{-\tau}^{0} |Y_{\tk} (\th)|^2 d \th \Big )  \no \\
& \qu + C\D  \sum_{k=0}^{n-1} \int_{-\tau}^{0} |Y_{\tk} (\th)|^{2+ \hat r}  d \th  + \sum_{k=0}^{n-1} \Big ( J_3(k) +  J_4(k) \Big ).
\end{align}
Let $ \hat p \in \left [2,   \frac{p}{1 + ( r \ve \hat r) /2} \right ]$ and 
$   \bar q = \hat p /2$.
Raising both sides to the power $\hat p$ and using the 
 elementary inequality
\begin{align*}
\Big | \sum_{k=1}^{n}  a_k \Big |^{\bar  q} \le  n^{\bar q-1}\sum_{k=1}^{n}  |a_k|^{\bar q }, \qu \bar q \ge 1, \; a_k \in \R^{n_1}.
\end{align*}
we  conclude from \eqref{TT_1} that
\begin{align*}
\frac{|Y(n\D )|^{\hat p }}{7^{\bar q-1}}=
\frac{|Y(n\D )|^{2 \bar q }}{7^{\bar q-1}}  & \le |\xi (0)|^{\bar q} +
C\D^{\bar q}  n^{\bar q -1}\sum_{k=0}^{n-1}  \Big (1+ |Y(\tk)|^{2\bar q} + \frac{1}{\tau} \int_{-\tau}^{0} |Y_{\tk} (\th)|^{2\bar q} d \th \Big )  \no \\
& \qu + C\D^{\bar q}  n^{\bar q -1}\sum_{k=0}^{n-1} \int_{-\tau}^{0} |Y_{\tk} (\th)|^{\bar q(2+ \hat r)}  d \th  + \Big | \sum_{k=0}^{n-1}  J_3(k) \Big |^{\bar q} + \Big | \sum_{k=0}^{n-1}  J_4(k) \Big |^{\bar q} \no \\
& \le C(T) + 
C T^{\bar q -1} \D \sum_{k=0}^{n-1}  \Big ( 1+  |Y(\tk)|^{2\bar q} + \frac{1}{\tau} \int_{-\tau}^{0} |Y_{\tk} (\th)|^{2\bar q} d \th \Big )  \no \\
& \qu + C T^{\bar q -1} \D\sum_{k=0}^{n-1} \int_{-\tau}^{0} |Y_{\tk} (\th)|^{\bar q(2+ \hat r)}  d \th  + \Big | \sum_{k=0}^{n-1}  J_3(k) \Big |^{\bar q}  + \Big | \sum_{k=0}^{n-1}  J_4(k) \Big |^{\bar q}. 
\end{align*}
Consequently,
\begin{align}\label{TT4}
& \frac{\E \left[ \sup_{0 \le n \le N} |Y(n\D)|^{2 \bar q} \right ] }{7^{\bar q-1}}   \\
&  \le C(T)  + C(T) \D \sum_{k=0}^{N-1} (1 + \E |Y(k\D)|^{2 \bar q}) \no \\
& +  C(T)\D\E \left [ \sum_{k=0}^{N-1}  \frac{1}{\tau} \intt |Y_{k\Delta} (\th)|^{2 \bar q} d \th \right ]
  +  C(T)\D\E \left [ \sum_{k=0}^{N-1}  \intt |Y_{k\Delta} (\th)|^{\bar q(2+\hat r)}  d \th \right ]  \no \\
&
+ \E \left[\sup_{0 \le n \le N} \Big | \sum_{k=0}^{n-1}  J_3(k) \Big |^{\bar q} \right ]
+ \E \left[\sup_{0 \le n \le N} \Big | \sum_{k=0}^{n-1}  J_4(k) \Big |^{\bar q} \right ].
\no
\end{align}
Applying \eqref{Eqq3.2} and the Jensen inequality gives
\begin{align}\label{Tmp60016}
& \intt |Y_{k\Delta} (\th)|^{2\bar q}  d \th  
 = \sum_{j=-m}^{-1}\int_{j\D}^{(j+1)\D} |Y_{k\Delta} (\th)|^{2\bar q}  d\th  \no \\
& = \sum_{j=-m}^{-1}\int_{j\D}^{(j+1)\D} 
\Big | \frac{(j+1)\D -\th}{\D} 
  Y((k+j)\D)  +  \frac{\th - j\D}{\D} 
  Y((k+j+1)\D) \Big |^{2\bar q}  d \th \no \\
& \le  \sum_{j=-m}^{-1}\int_{j\D}^{(j+1)\D} 
\Big ( \frac{(j+1)\D -\th}{\D} 
  |Y((k+j)\D)|^{2\bar q}  +  \frac{\th - j\D}{\D} 
 | Y((k+j+1)\D)|^{2\bar q} \Big )  d \th \no \\
& = \frac{\D}{2} \sum_{j=-m}^{-1}
\Big ( | Y((k+j)\D)|^{2\bar q}+ | Y((k+j+1)\D)|^{2\bar q}\Big ).
\end{align}
Thus, 
%
\begin{align}\label{Est_21}
& \sum_{k=0}^{N-1}  \frac{1}{\tau} \intt |Y_{k\Delta} (\th)|^{2\bar q} d \th
 \le  \frac{\D}{2\tau }   \sum_{j=-m}^{-1}  \sum_{k=0}^{N-1}
\Big ( | Y((k+j)\D)|^{2 \bar q}  +  | Y((k+j+1)\D)|^{2 \bar q}\Big )  \no \\
& \le  \frac{1}{m}  \sum_{j=-m}^{-1}  \sum_{k=-m}^{N-1} | Y(k\D)|^{2 \bar q}  
=\sum_{k=-m}^{N-1} |Y(k\D)|^{2 \bar q} 
 \le  m \| \xi \|^{2 \bar q}+ \sum_{k=0}^{N-1} |Y(k\D)|^{2 \bar q} .
\end{align}
Similarly, 
\begin{align}\label{Est_Zk1}
\sum_{k=0}^{N-1}  \frac{1}{\tau} \intt |Y_{k\Delta} (\th)|^{\bar q (2+\hat r)}  d \th
 \le \sum_{k=-m}^{N-1} |Y(k\D)|^{\bar q (2+\hat r)}\le 
 m \| \xi \|^{\bar q (2+\hat r)}+ \sum_{k=0}^{N-1} |Y(k\D)|^{\bar q (2+\hat r)} .
\end{align}
Inserting \eqref{Est_Zk1} and \eqref{Est_21} into \eqref{TT4}, one have
\begin{align}\label{TT8}
& \frac{\E \left[ \sup_{0 \le n \le N} |Y(n\D)|^{2 \bar q} \right ] }{7^{\bar q-1}}   \\
&  \le C(T)  (1 + \tau \| \xi \|^{\bar q (2+\hat r)} ) + C(T) \D \sum_{k=0}^{N-1} ( \E |Y(k\D)|^{2 \bar q} + \E |Y(k\D)|^{ \bar q (2 + \hat r)}) \no \\
& + \E \left[\sup_{0 \le n \le N} \Big | \sum_{k=0}^{n-1}  J_3(k) \Big |^{\bar q} \right ]
+ \E \left[\sup_{0 \le n \le N} \Big | \sum_{k=0}^{n-1}  J_4(k) \Big |^{\bar q} \right ]
\no \\
& \le C(T)  (1 +  \| \xi \|^{\bar q (2+\hat r)} )
+ \E \left[\sup_{0 \le n \le N} \Big | \sum_{k=0}^{n-1}  J_3(k) \Big |^{\bar q} \right ]
+ \E \left[\sup_{0 \le n \le N} \Big | \sum_{k=0}^{n-1}  J_4(k) \Big |^{\bar q} \right ], \no 
\end{align}
where \eqref{Bound_Xp312} has been used.
By virtue of the \BDG  inequality and the elementary inequality
$ab \le \eps a^2 + \frac{b^2}{4\eps}$, one obtain
\begin{align}\label{Eq_G51}
& \E \left[\sup_{0 \le n \le N} \Big | \sum_{k=0}^{n-1}  J_3(k) \Big |^{\bar q} \right ]
=  \E \left[\sup_{0 \le n \le N} \Big | \sum_{k=0}^{n-1} \lan Y(k \D), g(Y_{k \D}) \D B_{k\D} \ran \Big |^{\bar q} \right ]  \\
& \le 4 \sqrt{2}
\E \left [  \sum_{k=0}^{N-1} |\lan Y(k \D), g(Y_{k \D}) \ran |^2 \D \right ]^{\bar q /2} 
  \le 4 \sqrt{2}
\E \left [ \Big ( \sup_{0 \le k \le N} |Y(k \D)|^2 \Big )
 \sum_{k=0}^{N-1} | g(Y_{k \D}) |^2 \D \right ]^{\bar q /2} \no \\
& \le \frac{1}{2 \cdot 7^{\bar q -1}} \E \left [  \sup_{0 \le k \le N} |Y(k\D)|^{2 \bar q} \right ]
+C(\bar q) \E \left [ \sum_{k=0}^{N-1}| g(Y_{k \D}) |^{2 \bar q } \D\right ]. \no
\end{align}
According to \eqref{Bound_Xp312} and \eqref{Eq_B1},
  we have
\begin{align}\label{EQ_G61}
\E \left [ \sum_{k=0}^{N-1}| g(Y_{k \D}) |^{2 \bar q } \D\right ]
&\le C(\bar q) \E \left [ \sum_{k=0}^{N-1}\Big ( 1 +  | Y(k \D)) |^{2 \bar q}  +
\frac{1}{\tau}\intt |Y_{k \D} (\th)|^{\bar q (2+r) } d \th
 \Big )\D\right ]   \no \\
 & \le C(\bar q) \sum_{k=-m}^{N-1}\Big ( 1 +  \E| Y(k \D)) |^{\bar q (2+r)}
 \Big )\D
  \le C(\bar q,T) (1 + \| \xi \|^{\bar q (2+r)}).
\end{align}
Similarly,
\begin{align}\label{Eq_G81}
& \E \left[\sup_{0 \le n \le N} \Big | \sum_{k=0}^{n-1}  J_4(k) \Big |^{\bar q} \right ]
\le  \E \left[ \sup_{0 \le n \le N} n^{\bar q -1} \sum_{k=0}^{n-1}|g(Y_{k \D})|^{2 \bar q} \big ( |\D B_{k\D}|^2 + \D \big )^{ \bar q}  \right ] \no  \\
& \le C(\bar q) \E \left[  N^{\bar q -1} \sum_{k=0}^{N-1}|g(Y_{k \D})|^{2 \bar q} \big ( |\D B_{k\D}|^{2 \bar q} + \D ^{\bar q} \big )  \right ]  \no \\
& \le C(\bar q)N^{\bar q -1} \D^{\bar q}  \sum_{k=0}^{N-1} \E |g(Y_{k \D})|^{2 \bar q}
\le C(\bar q,T) (1 + \| \xi \|^{\bar q (2+r)}).
\end{align}
It then follows from \eqref{TT8}--\eqref{Eq_G81} that
\begin{align*}
&
\E \left[ \sup_{0 \le n \le N} |Y(n\D)|^{2 \bar q} \right ] 
\le \frac{1}{2} \E \left[ \sup_{0 \le n \le N} |Y(n\D)|^{2 \bar q} \right ] +  C(\bar q,T) (1 + \| \xi \|^{\bar q (2+r)}), 
\end{align*}  
which implies the assertion \eqref{Bound_XX242}. 
\par 
Now, we begin to establish \eqref{Bound_XX2422}.
 By virtue of \eqref{Eq932} and \eqref{PP2}, we deduce
 \begin{align*}
 |Z(t)|^{\hat p} 
 & \le 3^{\hat p-1} \left (    | Y(\tk)|^{\hat p}  + \D^{\hat p}   | f (Y_{\tk})|^{\hat p}  +    | g (Y_{\tk}) (B(t) - B (\tk))|^{\hat p}  \right ) \\
 & \le C(\hat p) \left (  1+ \D^{\hat p(1 - \varrho)}    \right )
 \left ( 1 +  | Y(\tk)|^{\hat p }  + \frac{1}{\tau}\intt | Y_{\tk}(\th)|^{\hat p } d \th  \right ) +  C(\hat p) | g (Y_{\tk}) (B(t) - B (\tk))|^{\hat p}  \no \\
 & \le  C(\hat p) \left ( 1+
 | Y(\tk)|^{\hat p }  + \frac{1}{\tau}\intt | Y_{\tk}(\th)|^{\hat p } d \th 
 +| g (Y_{\tk}) (B(t) - B (\tk))|^{\hat p}
  \right )
 \end{align*}
Thus,  applying \eqref{Bound_XX242} and the Doob martigale inequality gives
\begin{align}
& \E \left[ \sup_{0 \le t \le T}  |Z(t)|^{\hat p}  \right ] \no \\
& \le 
 C(\hat p) \left ( 1+
\E \left [ \sup_{0 \le k \le N}| Y(\tk)|^{\hat p }   \right ]+ 
\E \left [ \sup_{0 \le k \le N}
\frac{1}{\tau}\intt | Y_{\tk}(\th)|^{\hat p } d \th  \right ] 
 + \E \left[ \sup_{0 \le t \le T}  | g (Y_{\underline t  }) (B(t) - B (\underline t  ))|^{\hat p} \right ]
  \right ) \no \\ 
 & \le C(\hat p,T) 
 \left ( 1+
\E \left [ \sup_{-m  \le k \le N}| Y(\tk)|^{\hat p }   \right ]
 + \sum_{k=0}^{N}\E \left[ \sup_{\tk \le t \le \tkk}  | g (Y_{\tk  }) (B(t) - B (\tk ))|^{\hat p} \right ] \right )  \no \\
 & \le C(\hat p,T) (1 + \|\xi \|^{p}) + C(T)   \sum_{k=0}^{N}\E \left[  | g (Y_{\tk  }) \D B_{k \D } |^{\hat p} \right ] \no\\ 
  &\le C(\hat p,T) (1 + \|\xi \|^{p}) + C(T)  \D^{\hat p /2} \sum_{k=0}^{N}\E \left[  | g (Y_{\tk  })  |^{\hat p} \right ] \no \\
 & \le C(\hat p,T) (1 + \|\xi \|^{p}) N\D^{\hat p /2}   
 \le C(\hat p,T) (1 + \|\xi \|^{p}). 
\end{align}
Similarly, we can show that \eqref{Bound_XX24266}  also holds. 
Thus,  the proof is finished. 
\end{proof}
We now cite a known result as a lemma, the proofs of \eqref{Tem103},   \eqref{G99_10_1} and \eqref{G99_10} can be found in  \cite[Lemma 3.6, Lemm 3.8 and Theorem 3.3]{Li_Mao2024spa}, respectively.
\begin{lemma}\label{Lem904}
Let Assumptions \ref{AS2_Kha}, \ref{AS5_POLYNO} and \ref{AS3} hold. 
\begin{itemize}
  \item  If $ 2 + r  < q$, then for $\hat p \in \left [2, \frac{q}{1+ r/2} \right )$ and $\D \in (0,1]$, 
 \begin{align}\label{Tem103}
\sup_{0 \le t < \infty} \sup_{ -\tau \le \th \le 0} \E |Z_t(\th) - \bar Y_t(\th)|^{\hat p} \le C(\hat p,T) \D^{\hat p/2} ,
\end{align} 
\item If $ 2 + 3 r <q $, then for $ \hat p \in \left [ 2, \frac{q}{1+ 3r/2}  \right ) $
and $\D \in (0,1]$,
    \begin{align}
 \sup _{0 \le t \le T }\E |X(t) - Z( t)|^{\hat p}\le C(\hat p,T) \D^{\hat p/2}, \label{G99_10_1}     \\
 \sup _{0 \le t \le T }\E |X(t) - Y(\underline t)|^{\hat p}\le C(\hat p,T) \D^{\hat p/2}. \label{G99_10}
\end{align}
\end{itemize}
\end{lemma}
For any
$ \hat p \in \left [ 2, \frac{q}{1+ 3r/2}  \right ) $, we conclude from
 \eqref{G99_10_1} and \eqref{Tem103} that 
\begin{align}\label{Za12}
& \sup_{ 0 \le t \le T} \sup_{-\tau \le \th \le 0} \E \Big [|X_t (\th) - \bar Y_t (\th)|^{\hat p } \Big ]  \no \\
& \le C(\hat p) \Big ( \sup_{ 0 \le t \le T} \sup_{-\tau \le \th \le 0}\E \Big [|X_t (\th) - Z_t (\th)|^{\hat p } \Big ] +
\sup_{ 0 \le t \le T} \sup_{-\tau \le \th \le 0} \E \Big [|Z_t (\th) - \bar Y_t (\th)|^{\hat p } \Big ]    \Big )   \no \\
& \le 
C(\hat p) \Big ( \sup_{-\tau \le t \le T} \E \Big [|X (t) - Z (t)|^{\hat p } \Big ] +
\sup_{ 0 \le t \le T} \sup_{-\tau \le \th \le 0}\E \Big [|Z_t (\th) - \bar Y_t (\th)|^{\hat p } \Big ]    \Big )   
 \le 
 C(\hat p, T) \D^{\hat p/2 }.
\end{align}
%
 
 To establish the convergence order of the truncated EM segment sequence, we require a result stronger than \eqref{Tem103}, which is provided by Lemma \ref{Lem904_22}.
\begin{lemma}\label{Lem904_22}
Let  Assumptions \ref{AS2_Kha}, \ref{AS5_POLYNO} and \ref{AS3} hold
with 
\begin{align}
\label{cond42}
\big (2+r \big )\big (2 + (r \ve \hat r)\big) \le p <q.
\end{align}
Then for any $\D \in (0,1]$,
%
\begin{align}\label{Tmp11_01}
\sup_{0 \le t \le T} 
\E \Big [ \sup_{-\tau \le \th \le 0} |Z_t (\th) - \bar Y_t (\th)|^2 \Big ] 
\le C(T) \D^{1/2 \we 2 \varrho /r}.
\end{align}
\end{lemma}
\begin{proof}
Fix any $t \in [0,T ]$. 
\eqref{Eqq3.2} can be rewritten as the following 
\begin{align}\label{Z11}
Y_{\underline t}(\th) =  Y(\underline t + \underline \th) + \frac{\th - \underline \th}{\D} 
\left (  Y(\underline t + \underline \th + \D)   - Y(\underline t + \underline \th) \right), \qu \f \th \in [-\tau,0],
\end{align}
where $\underline \cdot $ is defined by  \eqref{Def1}.
Applying the elementary inequality implies 
\begin{align}\label{Z12}
&\E \Big [ \sup_{-\tau \le \th \le 0} |Z_t (\th) - \bar Y_t (\th)|^2 \Big ] 
= \E \Big [ \sup_{-\tau \le \th \le 0} |Z (t+\th) -  Y_{\underline t} (\th)|^2 \Big ] 
 \no \\
& \le 2 \E \Big [ \sup_{-\tau \le \th \le 0} |Z (t+\th) -  Y(\underline t + \underline \th)|^2 \Big ]
+ 2 \E \Big [ \sup_{-\tau \le \th \le 0} | Y(\underline t + \underline \th + \D)   - Y(\underline t + \underline \th) |^2 \Big ]
  \no \\
& \le  
  4 \underbrace{ \E \Big [ \sup_{-\tau \le \th \le 0} |\hat Y(\underline t + \underline \th +\D)- Y(\underline t + \underline \th)|^2 \Big ]}_{=:E_1}
+4 \underbrace{ \E \Big [ \sup_{-\tau \le \th \le 0} | Y(\underline t + \underline \th +\D)- \hat Y(\underline t + \underline \th +\D)|^2 \Big ] }_{=:E_2}
  \no \\
& \qu
+ 2 \underbrace{ \E \Big [ \sup_{-\tau \le \th \le 0} |Z (t + \th) -  Y(\underline t + \underline \th)|^2 \Big ] }_{=:E_3}, \qu  t \in [0,T].
\end{align}
\noindent $\mathbf{Step 1 : Estimate~  E_{1}}$.
We first prove that 
\begin{align}\label{tmp11_3}
E_1 =\E \Big [ \sup_{-\tau \le \th \le 0} |\hat Y(\underline t + \underline \th +\D)- Y(\underline t + \underline \th)|^2 \Big ]\le C(T) \D^{1/2}.
\end{align}
If $\underline t + \underline \th \le -\D  $, then by Assumption \ref{AS3}  and \eqref{Eqq3.2}, 
we have $E_1 \le C \D$.
If $\underline t + \underline \th \ge  0  $, then we conclude from  \eqref{TEM1} that 
\begin{align}\label{tmp11_4}
E_1& = \E \Big [ \sup_{-\tau \le \th \le 0} |f(Y_{\underline t + \underline \th})\D +
g(Y_{\underline t + \underline \th})\D B_{\underline t + \underline \th} 
|^2 \big ] \no \\
& \le 2 \D^2 \E \Big [ \sup_{-\tau \le \th \le 0} |f(Y_{\underline t + \underline \th})|^2 \big ]
+ 2 \E \Big [ \sup_{-\tau \le \th \le 0} |g(Y_{\underline t + \underline \th})\D B_{\underline t + \underline \th} |^2 \big ]. 
\end{align}
Applying \eqref{Eq_B11}  and the Jensen inequality yields
\begin{align}
 |f(\psi)|^2 & \le C \Big ( 1 + |\psi (0)|^{2+ 2r} + \Big (  \frac{1}{\tau}\intt | \psi(\th) |^{1 +r} d \th  \Big )^2\Big ) 
 \le C \Big ( 1 + |\psi (0)|^{2+ 2r} +  \frac{1}{\tau}\intt | \psi(\th) |^{2 +2r} d \th  \Big )\label{temp402-1}, \\
 |g(\psi)|^4 
& \le C \Big ( 1 + |\psi (0)|^{4+ 2r} +  \frac{1}{\tau}\intt | \psi(\th) |^{4 +2r} d \th  \Big ).\label{temp402-2}
\end{align}
Let $N= \lfloor T/\D \rfloor$. Condition \eqref{cond42} implies that  $\displaystyle 2+2r < 4 + 2 r \le \frac{p}{1 + (r \ve \hat r)/2}$. Thus,  
 we conclude from \eqref{temp402-1} and \eqref{Bound_XX242} of Lemma \ref{Lem3.12} that 
\begin{align}\label{tmp11_5}
\E \Big [ \sup_{0 \le k \le N} |f(Y_{k \D})|^2 \big ] 
& \le C \Big ( 1+  \E \Big [ \sup_{0 \le k \le N} |Y(k\D)|^{2+ 2r} \Big ] +
\E \Big [ \sup_{0 \le k \le N}\frac{1}{\tau} \int_{-\tau}^{0 }|Y_{k \D}(\th)|^{2+ 2r}  d \th \Big ] \Big  )  \no \\
& \le C \Big ( 1+  \E \Big [ \sup_{0 \le k \le N} |Y(k\D)|^{2+ 2r} \Big ] +
\E \Big [ \sup_{0 \le k \le N}|Y_{k \D}(\th)|^{2+ 2r}  \Big ] \Big  ) \le C(T).
\end{align}
Similarly, it follows from  condition \eqref{cond42} and \eqref{temp402-2} that 
\begin{align}\label{tmp11_52}
\E \Big [ \sup_{0 \le k \le N} |g(Y_{k \D})|^4 \big ] 
\le C(T).
\end{align}
Consequently, by using
 the \Holder and the Doob maximal inequalities, we can estimate 
\begin{align}\label{tmp11_6}
 \E \Big [ \sup_{-\tau \le \th \le 0} |g(Y_{\underline t + \underline \th})\D B_{\underline t + \underline \th} |^2 \big ] 
& \le \left (\E \Big [ \sup_{-\tau \le \th \le 0} |g(Y_{\underline t + \underline \th})|^4 \Big ] \right )^{1/2}
 \left (\E \Big [ \sup_{-\tau \le \th \le 0} |\D B_{\underline t + \underline \th} |^4 \Big ] \right )^{1/2} \no \\
&\le \frac{16}{9} \left (\E \Big [ \sup_{-m \le k  \le N} |g(Y_{k\D})|^4 \Big ] \right )^{1/2}
 \left (\E \Big [\sum_{i=-m}^{-1}|\D B_{\underline t + i\D} |^4 \Big ] \right )^{1/2} \no \\
 & \le C_{T} \left ( \sum_{i=-m}^{-1}\E |\D B_{\underline t + i\D} |^4  \right )^{1/2} 
 \le C(T) ( m \D^2 ) ^{1/2}
 =  \tau^{1/2} C(T) \D^{1/2}.
\end{align}
Inserting \eqref{tmp11_5} and \eqref{tmp11_6} into  \eqref{tmp11_4} gives  \eqref{tmp11_3}. 
\par
\noindent $\mathbf{Step 2 : Estimate~  E_{2}}$.
We now begin to prove 
\begin{align}\label{tmp18-2}
E_2 = \E \Big [ \sup_{-\tau \le \th \le 0} |\hat Y(\underline t + \underline \th +\D)-  Y(\underline t + \underline \th +\D)|^2 \Big ] \le C(T) \D^{2 \varrho /r}.
\end{align}
 Let $R(\D):= H^{-1}(c_4 \D^{-\varrho} )=K\D^{-\varrho/r} $.
Then we conclude from the definition of truncation that 
\begin{align*}
& |\hat Y(\tk) - Y(\tk)|^2= 
 \Bigg |\hat Y(\tk)  - R(\D) \frac{\hat Y(\tk)}{|\hat Y(\tk)|} \Bigg |^2\II_{ \{|\hat Y(\tk)| \ge R(\D) \} }  \\
& \le   2( |\hat Y(\tk)  |^2 + R(\D)^2 )\II_{ \{|\hat Y(\tk)| \ge R(\D) \} }
 \le 4|\hat Y(\tk)  |^2 \II_{ \{|\hat Y(\tk)| \ge R(\D) \} },  \; k \ge 0.
\end{align*}
 By \eqref{cond42},  the \Holder inequality and \eqref{Bound_XX24266}, we get 
\begin{align}\label{tmp18-1}
& \E \Big [ \sup_{-m \le k \le N} |  \hat Y(\tk) -  Y(\tk)|^2 \Big ] 
 \le 4 \E \Big [ \sup_{-m \le k \le N} |  \hat Y(\tk) |^2 \II_{ \{|\hat Y(\tk)| \ge R(\D) \} }\Big ] \no \\
 & \le 4 \left ( \E \Big [ \sup_{-m \le k \le N} |  \hat Y(\tk) |^4 \Big ]  \right )^{1/2}
  \left (\E \Big [ \sup_{-m \le k \le N}   \II_{ \{|\hat Y(\tk)| \ge R(\D) \} }\Big ]  \right )^{1/2} \no \\
  & \le 4 \left ( \E \Big [ \sup_{-m \le k \le N} |  \hat Y(\tk) |^4 \Big ]  \right )^{1/2}  \left ( \frac{\E \Big [ \sup_{-m \le k \le N} |  \hat Y(\tk) |^4 \Big ]}{R(\D)^4}   \right )^{1/2}  \no \\
 & \le C(T) \frac{1}{R(\D)^2} \le C(T) \D^{2\varrho/r},
\end{align}
which implies that \eqref{tmp18-2} holds.  Moreover, 
by virtue of \eqref{tmp11_3} and \eqref{tmp18-2}, we  conclude from the elementary inequality that 
\begin{align}\label{tmp11_32}
\E \Big [ \sup_{-\tau \le \th \le 0} | Y(\underline t + \underline \th +\D)-  Y(\underline t + \underline \th )|^2 \Big ] \le C(T) \D^{1/2 \we 2 \varrho /r}.
\end{align}
\noindent $\mathbf{Step 3 : Estimate~  E_{3}}$.
We begin to prove that 
\begin{align}\label{Teq_66}
E_3=\E \Big [ \sup_{-\tau \le \th \le 0} |Z (t + \th) -  Y(\underline t + \underline \th)|^2 \Big ]
\le C(T) \D^{1/2 \we 2 \varrho /r}.
\end{align}
Clearly, $ t \in [ \underline t, \underline t + \D) $ and $ \th \in [ \underline \th, \underline \th + \D) $. 
Thus, $ 0 \le t + \th  - (\underline t + \underline \th ) < 2 \D.$ 
Assume that $\underline t = k \D $ and  $\underline \th = i \D$. According to the range of $t + \th$,
we divide the proof of \eqref{Teq_66} into five cases. 
\begin{itemize}
\item [Case 1.] If $t+\th \in [(k+i)\D, (k+i+1)\D) \subset [-\tau,0)$, then  $\underline t +\underline \th = \underline {t + \th} $. From Assumption \ref{AS3}, we get
\begin{align*}
|Z(t+\th) - Y(\underline t + \underline \th )|^2= |\xi(t + \th ) - \xi(\underline {t +\th})|^2 \le C \D,
\end{align*}
which implies that \eqref{Teq_66} holds. 
\item [Case 2.] If $t+\th \in [(k+i+1)\D, (k+i+2)\D) \subset [-\tau,0)$, then  $\underline t +\underline \th  +\D= \underline {t + \th} $. Thus, 
\begin{align*}
&|Z(t+\th) - Y(\underline t + \underline \th )|^2 
 \le 2 |\xi(t+\th) - \xi(\underline{t+\th}) |^2 + 2 |\xi( \underline t+ \underline \th + \D ) - \xi(\underline t + \underline \th) |^2 \le C\D,
\end{align*}
which implies that \eqref{Teq_66} also holds. 
\item [Case 3.] If $t+\th \in [(k+i)\D, (k+i+1)\D) \subset [0, \infty)$, then 
the Doob martingale inequality gives 
\begin{align}
\label{Teq25_1}
&\E \Big [ \sup_{-\tau \le \th \le 0} |B (t + \th) -  B(\underline {t  +\th})|^4 \Big ]
 =\E \Big [ \sup_{t-\tau  \le s \le t} |B (s) -  B(\underline {s})|^4 \Big ] \no \\
& \le \E \Big [ \sup_{(k-m)\D  \le s \le (k+1)\D} |B (s) -  B(\underline {s})|^4 \Big ] 
\le  \sum_{i=k-m}^{k} \E \Big [  \sup_{i\D  \le s \le (i+1)\D} |B (s) -  B(i \D)|^4 \Big ] \no \\
& \le \left (\frac{4}{3}\right )^{4}   \sum_{i=k-m}^{k} \E \Big [  |B ((i+1)\D) -  B(i\D)|^4 \Big ] \le C\D. 
\end{align}
Moreover, 
\begin{align}
\label{Teq25_2}
&Z (t + \th) - Z(\underline {t + \th}) 
=Z(t+\th) - Y(\underline t + \underline \th ) \no \\
&=
f(Y_{(k+i)\D} )\Big (t + \th - (k+i)\D \Big )+g(Y_{(k+i)\D} )\Big ( B(t + \th) - B((k+i)\D) \Big ).
\end{align}
By the \Holder inequality and \eqref{tmp11_5} and \eqref{Teq25_1} 
, we conclude from \eqref{Teq25_2} and \eqref{tmp11_52} that 
\begin{align}
\label{Teq25-4}
& \E \Big [ \sup_{-\tau \le \th \le 0} |Z (t + \th) -  Y(\underline t + \underline \th)|^2 \Big ] \\
& \le 2 \D^2 \E \Big [ \sup_{-m \le i \le N} |f(Y_{(k+i)\D})|^2 \Big ]
+ 2 \E \Big [ \sup_{-\tau \le \th \le 0} |g(Y_{\underline {t  +\th}})
 [B (t + \th) -  B(\underline {t  +\th}) ]|^2 \Big ] \no \\
& \le C(T) \D^2 +   \left (\E \Big [ \sup_{-\tau \le \th \le 0} |g(Y_{\underline {t + \th}})|^4 \Big ] \right )^{1/2}
 \left (\E \Big [ \sup_{-\tau \le \th \le 0} |B (t + \th) -  B(\underline {t  +\th})|^4 \Big ] \right )^{1/2} 
 \le C(T) \D^{1/2}.  \no 
\end{align}
\item [Case 4.] If $t+\th \in [(k+i+1)\D, (k+i+2)\D) \subset [\D, \infty)$, then 
\begin{align*}
Z (t + \th) -  Y(\underline t + \underline \th)
= Z (t + \th)- Z(\underline {t + \th}) + Y(\underline {t} + \underline {\th}+\D)-  Y(\underline t + \underline \th).
\end{align*}
Thus, 
\begin{align*}
& \E \Big [ \sup_{-\tau \le \th \le 0} |Z (t + \th) -  Y(\underline t + \underline \th)|^2 \Big ] \no \\
& \le 2 \E \Big [ \sup_{-\tau \le \th \le 0} |Z (t + \th) -  Z(\underline {t + \th})|^2 \Big ] 
+  2 \E \Big [ \sup_{-\tau \le \th \le 0} |Y(\underline {t} + \underline {\th}+\D)-  Y(\underline t + \underline \th)|^2 \Big ].
\end{align*}
By \eqref{Eq932}, we have 
\begin{align*}
Z (t + \th) -  Y(\underline {t + \th})= f(Y_{\underline {t + \th}} )\Big (t + \th - \underline {t + \th} \Big )+g(Y_{\underline {t + \th}} )\Big ( B(t + \th) - B(\underline {t + \th}) \Big ).
\end{align*}
Following the proof of \eqref{Teq25-4}, we deduce that
\begin{align*}
\E \Big [ \sup_{-\tau \le \th \le 0} |Z (t + \th) -  Z(\underline {t + \th})|^2 \Big ] 
=\E \Big [ \sup_{-\tau \le \th \le 0} |Z (t + \th) -  Y(\underline {t + \th})|^2 \Big ] 
\le C(T) \D^{1 /2}.
\end{align*}
In addition, in light of \eqref{tmp11_32} yields
\begin{align*}
& \E \Big [ \sup_{-\tau \le \th \le 0} |Z (t + \th) -  Y(\underline t + \underline \th)|^2 \Big ] 
\le C(T) \D^{ 1/2 \we 2 \varrho /r}. 
\end{align*}
\item [Case 5.] If $t+\th \in [(k+i+1)\D, (k+i+2)\D) \subset [0, \D)$, then 
\begin{align*}
Z (t + \th) -  Y(\underline t + \underline \th)
= Z (t + \th) - Z(0) + \xi (0) - \xi(-\D).
\end{align*}
Thus, 
\begin{align*}
& \E \Big [ \sup_{-\tau \le \th \le 0} |Z (t + \th) -  Y(\underline t + \underline \th)|^2 \Big ]  \\
&  \le 2 \E \Big [ \sup_{-\tau \le \th \le 0} |f(\xi)(t+\th) + g (\xi) B(t+\th)|^2 \Big ]
+ 2 | \xi(0) - \xi(-\D)|^2 
\le C \D. 
\end{align*}
\end{itemize}
Combining the above five cases together, we prove that \eqref{Teq_66} holds. 
Thus, the proof is completed.
\end{proof}
We now investigate the $L^2$-error between the  auxiliary process $Z(t)$ and the exact solution $X(t)$.
\begin{theorem}\label{Th942}

 Let  Assumptions \ref{AS2_Kha}, \ref{AS5_POLYNO} and \ref{AS3} hold
with 
\begin{align}\label{cond726}
\displaystyle 2 + 4 r  \le (p-r) \we  \frac{p}{1 + (r \ve \hat r)/2}.
\end{align}
Let
$\displaystyle \hat p = ( p-r) \we  \frac{p}{1 + (r \ve \hat r)/2} $ and 
$R(\D)= H^{-1}(c_4 \D^{-\varrho} )=C\D^{-\frac{1}{3r}}.$
Then for any $\D \in (0,1]$, the truncated process $Z(t)$ defined by \eqref{Eq932} 
has the property that
\begin{align}
\E \left [ \sup_{0 \le t \le T}|X(t) - Z(t)|^{2} \right ]&\le C(T) 
\D^{\frac{1}{2}  \we ( \frac{p}{3r} - \frac{1}{\hat p -2}  )}, \qu  \label{Con_rate}  \\
\E \left [ \sup_{0 \le t \le T}|X(t) - Y( \underline t ) |^{2} \right ]&\le C(T) 
\D^{\frac{1}{2}  \we ( \frac{p}{3r} - \frac{1}{\hat p -2}  )}, \label{Con_rate22}
\end{align}
where the constant $C(T)$ depends on $T$ but is independent of $\D$.
\end{theorem}
\begin{proof}
Let $\nu_{R, \D} := \tau_R \wedge \rho_{R,\D}$, with $\tau_R$ and $\rho_{R,\D}$ defined in \eqref{Eq9434} and \eqref{Eq94322}, respectively.
Let  
$R(\D)= H^{-1}(c_4 \D^{-\varrho} )=C\D^{-\varrho/r} $ and  
$\nu_\D: =\tau_{R(\D)} \we \rho_{R(\D),\D} $.
We first show that 
\begin{align}\label{G96}
\E \left [ \sup_{0 \le t \le T } |X(t \we \nu_\D) - Z(t \we \nu_\D)|^2 \right ] \le  C(T) \D^{} ,
\end{align}
By  
\eqref{Eq0} and  \eqref{Eq9444}, we have 
\begin{align*}
X(t \we \nu_\D)-Z(t \we \nu_\D)=\int_{0}^{t \we \nu_\D} \Big ( f(X_s) - f(\bar Y_s) \Big ) ds
+ \int_{0}^{t \we \nu_\D} \Big ( g(X_s) - g(\bar Y_s) \Big ) dB(s), \qu \f t \in [0, T].
\end{align*}
The \Ito formula gives
\begin{align*}
| X(t \we \nu_\D)-Z(t \we \nu_\D)|^2 & = \int_{0}^{t \we \nu_\D} \Big (
2 \lan X(s) - Z(s), f(X_s) - f(\bar Y_s) \ran + |g(X_s) - g(\bar Y_s)|^2 
 \Big ) ds + \hat M(t\we \nu_\D), 
\end{align*}
where 
\begin{align*}
\hat M(t) = 2\int_{0}^{t}
(X(s) - Z(s))^T (g(X_s) - g(\bar Y_s))   dB(s) .
\end{align*}
From Assumption \ref{AS2_Kha}, and noting that $ 2 ab \le a^2 + b^2$, we obtain
\begin{align}\label{G32}
& | X(t \we \nu_\D)-Z(t \we \nu_\D)|^2  \\  
& \le \int_{0}^{t \we \nu_\D} \Big (
2 \lan X(s) - Z(s), f(X_s) - f(Z_s) \ran + 2|g(X_s) - g(Z_s)|^2  \Big ) ds  \no \\
& \qu + \int_{0}^{t \we \nu_\D} \Big (
2 \lan X(s) - Z(s), f(Z_s) - f(\bar Y_s) \ran + 2|g(Z_s) - g(\bar Y_s)|^2  \Big ) ds + \hat M(t \we \nu_\D) \no \\
& \le \int_{0}^{t \we \nu_\D} \Big (
2 \lan X(s) - Z(s), f(X_s) - f(Z_s) \ran + 2|g(X_s) - g(Z_s)|^2  \Big ) ds  \no \\
& \qu +  \int_{0}^{t \we \nu_\D}
|X(s) - Z(s)|^2 ds 
 +   2 \int_{0}^{t \we \nu_\D} \Big ( |f(Z_s) - f(\bar Y_s)|^2+|g(Z_s) - g(\bar Y_s)|^2  \Big ) ds
+ \hat M(t \we \nu_\D) \no \\
&  \le -\a_1  \int_{0}^{t \we \nu_\D} 
|X(s) - Z(s)|^2 ( |X(s)|^{\hat r} + |Z(s)|^{\hat r}) ds  \no \\
& \qu +
\a_2 \int_{0}^{t \we \nu_\D} \Big (  \frac{1}{\tau} \intt |X_s( \th) - Z_s(\th)|^2 ( |X_s(\th)|^{\hat r} + |Z_s(\th)|^{\hat r} )    d\th  \Big )
ds \no \\
& \qu + (1+\a_0) \int_{0}^{t \we \nu_\D} |X(s) - Z(s)|^2 ds +
\a_0 \int_{0}^{t \we \nu_\D} \Big (  \frac{1}{\tau} \intt |X_s( \th) - Z_s(\th)|^2 d\th  \Big ) d s
 \no \\
& \qu  +   2 \int_{0}^{t \we \nu_\D} \Big ( |f(Z_s) - f(\bar Y_s)|^2+|g(Z_s) - g(\bar Y_s)|^2  \Big ) ds
+ \hat M(t \we \nu_\D) . \no
\end{align}
It follows from Assumptions \ref{AS3} and \eqref{Eq932} that
\begin{align}\label{G98}
& \frac{1}{\tau} \int_{0}^{t \we \nu_\D} \Big ( \intt | X(s+\th) - Z(s+\th)|^2
( 1 + |X(s+\th)|^{\hat r} + |Z(s+\th)|^{\hat r} ) 
 d \th \Big ) ds \no \\
& = \frac{1}{\tau}  \intt \Big (\int_{0}^{t \we \nu_\D}  
| X(s+\th) - Z(s+\th)|^2 ( 1 + |X(s+\th)|^{\hat r} + |Z(s+\th)|^{\hat r} )  ds  \Big ) d \th \no \\
&  \le \frac{1}{\tau}  \intt \Big (\int_{-\tau}^{t \we \nu_\D}  
| X(u) - Z(u)|^2( 1 + |X(u)|^{\hat r} + |Z(u)|^{\hat r} ) du  \Big ) d \th \no \\
&
=  \int_{0}^{t \we \nu_\D}  | X(u) - Z(u)|^2 ( 1 + |X(u)|^{\hat r} + |Z(u)|^{\hat r} )du .
\end{align}
Inserting \eqref{G98} into \eqref{G32} and using $ \a_1 > \a_2$, we have 
\begin{align*}
& | X(t \we \nu_\D)-Z(t \we \nu_\D)|^2  \\  
& \le (1+ 2 \a_0)  \int_{0}^{t \we \nu_\D}
|X(s) - Z(s)|^2 ds 
 +   2 \int_{0}^{t \we \nu_\D} \Big ( |f(Z_s) - f(\bar Y_s)|^2+|g(Z_s) - g(\bar Y_s)|^2  \Big ) ds
+ \hat M(t \we \nu_\D) \\
& \le  (1+ 2 \a_0)   \int_{0}^{t \we \nu_\D}
|X(s) - Z(s)|^2 ds +
2 \int_{0}^{t } \Big ( |f(Z_s) - f(\bar Y_s)|^2+|g(Z_s) - g(\bar Y_s)|^2  \Big )\II_{ \{s \le \nu_\D \}} ds+ \hat M(t \we \nu_\D) \\
& \le (1+ 2 \a_0)    \int_{0}^{t \we \nu_\D}
|X(s) - Z(s)|^2 ds +
C \int_{0}^{t } |Z_s(0) - \bar Y_s (0)|^2 ( 1+ |Z_s(0)|^{2r} +|\bar Y_s (0)|^{2r} )  \II_{ \{s \le \nu_\D \}} ds \\
& \qu + C \int_{0}^{t }\Big ( \frac{1}{\tau}\intt  |Z_s(\th) - \bar Y_s (\th)|^2 ( 1+ |Z_s(\th)|^{2r} +|\bar Y_s (\th)|^{2r} )  \II_{ \{s \le \nu_\D \}}  \Big )ds + \hat M(t \we \nu_\D) \\
& \le  (1+ 2 \a_0)   \int_{0}^{t \we \nu_\D}
|X(s) - Z(s)|^2 ds +
C \int_{0}^{T} |Z_s(0) - \bar Y_s (0)|^2 ( 1+ |Z_s(0)|^{2r} +|\bar Y_s (0)|^{2r} )   ds \\
& \qu + C \int_{0}^{T}\Big ( \frac{1}{\tau}\intt  |Z_s(\th) - \bar Y_s (\th)|^2 ( 1+ |Z_s(\th)|^{2r} +|\bar Y_s (\th)|^{2r} ) d\th   \Big )ds + \hat M(t \we \nu_\D),
\end{align*}
where Assumption \ref{AS5_POLYNO} has been used. 
Thus, 
\begin{align}\label{G41}
& \E \left [ \sup_{0 \le u \le t \we \nu_\D} |X(u) - Z(u)|^2 \right ]  \no \\
& \le  C \E \left [\int_{0}^{t \we \nu_\D} |X(s) - Z(s)|^2 ds  \right ]+
C \int_{0}^{T} \E \left [ |Z_s(0) - \bar Y_s (0)|^2 ( 1+ |Z_s(0)|^{2r} +|\bar Y_s (0)|^{2r} )  \right ] ds
 \no \\
& \qu
+ C \int_{0}^{T} \E \Big [ \frac{1}{\tau}\intt  |Z_s(\th) - \bar Y_s (\th)|^2 ( 1+ |Z_s(\th)|^{2r} +|\bar Y_s (\th)|^{2r} )  d \th  \Big ]ds
 + \E \left [ \sup_{0 \le u \le t \we \nu_\D} \hat M(u) \right ]   .
\end{align}
Let 
  $ 2 + 4 r \le p < q$,  then $\displaystyle \frac{2p}{p-2r} \le \frac{p}{1+r/2}$.
By the \Holder inequality and the Lyapunov inequality as well as \eqref{Tem103}, we have 
\begin{align}\label{G9333_10}
&\E \Big  [ \frac{1}{\tau}\intt |Z_s(\th)- \bar Y_s(\th)|^2 ( 1 + |Z_s(\th)|^{2r} + |\bar Y_s(\th)|^{2r})   d \th \Big ] \no \\
& =    \frac{1}{\tau}\intt \E  \Big  [|Z_s(\th)- \bar Y_s(\th)|^2 ( 1 + |Z_s(\th)|^{2r} + |\bar Y_s(\th)|^{2r})   \Big ] d \th    \no \\
& \le C \frac{1}{\tau}\intt \left ( \E  |Z_s(\th)- \bar Y_s(\th) |^{\frac{2p}{p-2r}} \right )^{\frac{p-2r}{p}}
 \left (\E \big (1 + |Z_s( \th)|^{p} + | \bar Y_s( \th)|^{p}\big ) \right )^{\frac{2r}{p}}   d \th \no \\
 & \le C \frac{1}{\tau}\intt \left (\sup_{0 \le s \le T} \sup_{-\tau \le \th \le 0}  \E  |Z_s(\th)-  \bar Y_s(\th) |^{\frac{2p}{p-2r}} \right )^{\frac{p-2r}{p}}
 \left (1 + \sup_{- \tau \le s \le T} \E |Z(s)|^{p} + \sup_{- \tau \le s \le T} \E| \bar Y(s)|^{p} \right )^{\frac{2r}{p}}   d \th  \no \\
 & \le C(T) \D^{},
\end{align}
where \eqref{Bound_Xp312} and  \eqref{Bound_Xp312_29} have been used. 
Similarly, we also can show that 
\begin{align}\label{Temp108}
\E \left [ |Z_s(0) - \bar Y_s (0)|^2 ( 1+ |Z_s(0)|^{2r} +|\bar Y_s (0)|^{2r} )  \right ] \le C(T) \D^{}, \qu 0 \le s \le T.
\end{align}
Inserting \eqref{G9333_10} and \eqref{Temp108} into \eqref{G41} yields 
\begin{align}\label{G43}
& \E \left [ \sup_{0 \le u \le t \we \nu_\D} |X(u) - Z(u)|^2 \right ]  
 \le C \int_{0}^{t} \E \left [ \sup_{0 \le u \le s \we \nu_\D} |X(u) - Z(u)|^2 \right ] ds+ \E \left [ \sup_{0 \le u \le t \we \nu_\D} \hat M(u) \right ] +  C(T) \D .
\end{align}
It then follows from the \BDG inequality and the elementary inequality that
\begin{align}\label{G45}
 \E \left [ \sup_{0 \le u \le t \we \nu_\D} \hat M(u) \right ] 
& \le  8 \sqrt{2}  \E \left [ \int_{0}^{t \we \nu_\D}
| X(s) - Z(s)|^2 | g(X_s) - g(\bar Y_s)|^2 ds \right ]^{1/2} \no \\
& \le  8 \sqrt{2}  \E \left [ \sup_{0 \le u \le t \we \nu_\D } | X(u) - Z(u)|^2
\int_{0}^{t \we \nu_\D}| g(X_s) - g(\bar Y_s)|^2 ds  
\right ]^{1/2} \no \\
& \le \frac{1}{2} \E \left [ \sup_{0 \le u \le t \we \nu_\D}  | X(u) - Z(u)|^2 \right ]
+ C \E  \int_{0}^{t \we \nu_\D}| g(X_s) - g(\bar Y_s)|^2 ds  \no \\
& \le \frac{1}{2} \E \left [ \sup_{0 \le u \le t \we \nu_\D}  | X(u) - Z(u)|^2 \right ]
+ C \int_{0}^{T} \E | g(X_s) - g(\bar Y_s)|^2 ds.
\end{align}
Inserting \eqref{G45} into \eqref{G43} gives that 
\begin{align}\label{G4443}
& \E \left [ \sup_{0 \le u \le t \we \nu_\D} |X(u) - Z(u)|^2 \right ]   \\
 & \le C \int_{0}^{t} \E \left [ \sup_{0 \le u \le s \we \nu_\D} |X(u) - Z(u)|^2 \right ] ds+  C \int_{0}^{T} \E | g(X_s) - g(\bar Y_s)|^2 ds +  C(T) \D . \no
\end{align}
 Again by Assumption \ref{AS5_POLYNO} and \eqref{G94}, we arrive at 
\begin{align}\label{G91}
 \E | g(X_s) - g(\bar Y_s)|^2 
& \le  C \E \Big  [  |X(s) - Y(\underline s) |^2( 1 + |X(s)|^{r} + |Y(\underline s)|^r ) \Big ] \no \\
& \qu + C \E \Big  [ \frac{1}{\tau}\intt |X_s(\th)- \bar Y_s(\th)|^2 ( 1 + |X_s(\th)|^r + |\bar Y_s(\th)|^r)   d \th \Big ].  
\end{align}
Moreover, 
  $ 2 + 4 r \le p < q$ implies that  $\displaystyle \frac{2p}{p-r} \le \frac{p}{1+3r/2}.$
By the \Holder inequality,  \eqref{Bound_Xp} of Lemma
 \ref{lem2.11} and  \eqref{Bound_Xp312} of Lemma  \ref{Lem3.12} as well as \eqref{G99_10} of Lemma \ref{Lem904}, we have 
\begin{align}\label{G92}
&\E \Big  [  |X(s) - Y(\underline s) |^2 \big (1+ |X(s)|^{r} + |Y(\underline s)|^r \big )\Big ] \no \\
& \le \left ( \E  |X(s) - Y(\underline s) |^{\frac{2p}{p-r}} \right )^{\frac{p-r}{p}}
 \left (\E \big (1 + |X(s)|^{p} + |Y(\underline s)|^{p}\big ) \right )^{\frac{r}{p}} 
\le C(T)\D.
\end{align}
By the \Holder  and the Lyapunov inequalities as well as \eqref{Za12}, we have 
\begin{align}\label{G9333}
&\E \Big  [ \frac{1}{\tau}\intt |X_s(\th)- \bar Y_s(\th)|^2 ( 1 + |X_s(\th)|^r + |\bar Y_s(\th)|^r)   d \th \Big ] \no \\
& =    \frac{1}{\tau}\intt \E  \Big  [|X_s(\th)- \bar Y_s(\th)|^2 ( 1 + |X_s(\th)|^r + |\bar Y_s(\th)|^r)   \Big ] d \th    \no \\
& \le C \frac{1}{\tau}\intt \left ( \E  |X_s(\th)- \bar Y_s(\th) |^{\frac{2p}{p-r}} \right )^{\frac{p-r}{p}}
 \left (\E \big (1 + |X_s( \th)|^{p} + | \bar Y_s( \th)|^{p}\big ) \right )^{\frac{r}{p}}   d \th \no \\
 & \le C \frac{1}{\tau}\intt \left (\sup_{0 \le s \le T} \sup_{-\tau \le \th \le 0}  \E  |X_s(\th)-  \bar Y_s(\th) |^{\frac{2p}{p-r}} \right )^{\frac{p-r}{p}}
 \left (1 + \sup_{- \tau \le s \le T} \E |X(s)|^{p} + \sup_{- \tau \le s \le T} \E| \bar Y(s)|^{p} \right )^{\frac{r}{p}}   d \th  \no \\
 & \le C(T) \D^{},
\end{align}
where \eqref{Bound_Xp} and  \eqref{Bound_Xp312} have been used. 
Thus, inserting \eqref{G92}, \eqref{G9333} into \eqref{G91} gives that 
\begin{align}\label{G9991}
 \E | g(X_s) - g(\bar Y_s)|^2 
 & \le  C(T) \D .
\end{align}
Inserting \eqref{G9991} into \eqref{G4443} gives 
\begin{align}\label{G95}
& \E \left [ \sup_{0 \le u \le t \we \nu_\D} |X(u) - Z(u)|^2 \right ] 
 \le C \int_{0}^{t} \E \left [ \sup_{0 \le u \le s \we \nu_\D} |X(u) - Z(u)|^2 \right ] ds+  C(T) \D^{}  .
\end{align}
Then applying the Gronwall inequality to the above \eqref{G95}, we get the desired assertion \eqref{G96}.
\par
Set $\displaystyle \hat p =( p-r) \we  \frac{p}{1 + (r \ve \hat r)/2} $.
For any $\delta >0$,
by the Young inequality and \eqref{Bound_XX3} as well as \eqref{Bound_XX2422}, we have 
\begin{align}\label{eq912_1}
& \E \left [ \sup_{ 0 \le t \le T} |X(t) - Z(t)|^2 \right ] \no \\
& \le \E \left [  \sup_{ 0 \le t \le T} |X(t) - Z(t)|^2 \II_{\{\tau_R >T \; \textrm{and} \; \rho_{R,\D} >T \} }     \right ]
+ \frac{2 \delta }{\hat p } \E \left [ \sup_{ 0 \le t \le T} |X(t) - Z(t)|^2 \right ] \no \\
& \qu + \frac{\hat p-2}{\hat p \delta^{2/(\hat p-2)}}\mathbb{P}(  \tau_R \le T \; \textrm{or} \; \rho_{R,\D} \le T  ) \no \\
& \le \E \left [  \sup_{ 0 \le t \le T} |X (t \we \nu_{\D}) -Z (t \we \nu_{\D})  |^2   \right ]
+ 
\frac{2 C}{ \hat p}\delta + 
 \frac{\hat p-2}{ \hat p \delta^{2/(\hat p-2)}}\mathbb{P}(  \tau_R \le T \; \textrm{or} \; \rho_{R,\D}\le T  ).
\end{align}
Set $\d =\D^{1/2}$ and  $\varrho= \frac{1}{3}$ , then  by 
 \eqref{Eq95} and \eqref{Eq952}, we have
\begin{align}\label{uu2}
 \frac{\hat p-2}{ \hat p \delta^{2/(\hat p-2)}}\mathbb{P}(  \tau_R \le T \; \textrm{or} \; \rho_{R,\D}\le T  )
\le C(T)\D^{\frac{ p}{3r} - \frac{1}{\hat p -2}} .
\end{align}
Inserting  \eqref{G96} and \eqref{uu2} into   \eqref{eq912_1} gives that 
\begin{align*}
 \E \left [ \sup_{ 0 \le t \le T} |X(t) - Z(t)|^2 \right ] 
& \le C(T) \D^{1/2} + C(T) \D + 
C(T)\D^{\frac{ p}{3r} - \frac{2}{\hat p -r}} 
 \le  C(T) \D^{\frac{1}{2} \we \left ( \frac{ p}{3r} - \frac{1}{\hat p -2} \right )} ,
\end{align*}
which is  
\eqref{Con_rate}. 
\par
  We now prove \eqref{Con_rate22}.
By \eqref{tmp11_5} and   \eqref{tmp11_52}, we have
\begin{align}\label{tmp419_1}
\E \Big [ \sup_{0 \le t \le T} |f(Y_{ \underline t})|^2 \big ]  \ve 
\E \Big [ \sup_{0 \le t \le T} |g(Y_{ \underline t})|^4 \big ] 
\le C(T).
\end{align}
Thus, by   \eqref{Eq932} and \eqref{tmp419_1}, we have 
\begin{align}\label{eq419_3}
\E \left [ \sup_{0 \le t \le T}|Z(t) - Y( \underline t ) |^{2} \right ]
& = \E \Big [ \sup_{0 \le t \le T} |f(Y_{\underline t })( t - \underline t) +
g(Y_{\underline t})( B(t) - B{(\underline t)}) 
|^2 \big ] \no \\
& \le 2 \D^2 \E \Big [ \sup_{0 \le t \le T} |f(Y_{\underline t })|^2 \big ]
+ 2 \E \Big [ \sup_{0 \le t \le T} |g(Y_{\underline t })( B(t) - B{(\underline t)})  |^2 \big ] \no \\  
& \le C(T)\D^2 + \left (\E \Big [ \sup_{0 \le t \le T} |g(Y_{\underline t })|^4 \Big ] \right )^{1/2}
 \left (\E \Big [ \sup_{0 \le t \le T} |(B(t) - B({\underline t})  |^4 \Big ] \right )^{1/2} \no \\& \le C(T)\D^2 +C(T) \D^{1/2} \le C(T) \D^{1/2}. 
\end{align}
By \eqref{Con_rate}  and \eqref{eq419_3}, we get decried  \eqref{Con_rate22}.
Thus, the proof is finished. 
\end{proof}  
Theorem \ref{Th943} establishes the strong $L^2$-convergence order of the segment sequence generated by the truncated EM method.
\begin{theorem}\label{Th943}
 Let  Assumptions \ref{AS2_Kha}, \ref{AS5_POLYNO} and \ref{AS3} hold
with 
\begin{align}\label{rem2}
\big (2+2r \big )\big (2 + (r \ve \hat r)\big) + r \le p <q.
    \end{align}  
Let 
 $\displaystyle \hat p =  ( p-r) \we  \frac{p}{1 + (r \ve \hat r)/2} $ and  $\displaystyle\hat \g := \frac{1}{2} \we \frac{2}{3r} \we ( \frac{p}{3r} - \frac{1}{\hat p -2}  )$ .
Then for any  $\D \in (0,1]$, the truncated EM scheme \eqref{TEM1} by setting 
$R(\D)=C\D^{-\frac{1}{3r}}$
has the property that 
\begin{align}\label{Con_rate2}
 \E\|X_T - \bar Y_T \|^{2} 
\le C(T)
 \D^{\hat \g },
\end{align}
where  
 the constant $C(T)$ depends on $T$ but is independent of $\D$.
\end{theorem}
\begin{proof}
Clearly, condition \eqref{rem2} implies \eqref{cond42} and \eqref{cond726}. Thus,  
we conclude from
  Lemma \ref{Lem904_22} and Theorem \ref{Th942}  that
%
\begin{align*}
& \E\|X_T - \bar Y_T \|^{2} 
\le  2 \E \left [ \sup_{-\tau \le \th \le 0} |X(T+\th) - Z_T(\th)|^2 \right ]
+ 2 \E \left [ \sup_{-\tau \le \th \le 0} | Z_T(\th) - \bar Y_T (\th)  |^2 \right ]
   \\
& \le 2 \E \left [ \sup_{0 \le t \le T} |X(t) - Z(t)|^2 \right ]
+ 2 \E \left [ \sup_{-\tau \le \th \le 0} | Z_T(\th) - \bar Y_T (\th)  |^2 \right ] \\
&  \le C(T) \D^{\frac{1}{2} \we \left ( \frac{ p}{3r} - \frac{1}{\hat p -2} \right )}
+ C(T)  \D^{1/2 \we 2 \varrho /r} 
\le C(T) \D^{\frac{1}{2} \we \frac{2}{3r} \we ( \frac{p}{3r} - \frac{1}{\hat p -2}  )} ,
\end{align*}
which completes the proof.
\end{proof}
\begin{remark}\label{rem1}
For certain specified parameters, we provide specific convergence orders of truncated EM scheme \eqref{TEM1}.  
%
     If $r= \hat r = 1$ and    
$q> p = 13$, then $ \hat p =  ( p-r) \we  \frac{p}{1 + (r \ve \hat r)/2} = \frac{26}{3}$ 
and mean-square convergence order $\hat \g = 
\frac{1}{2} \we \frac{2}{3} \we ( \frac{13}{3} - \frac{1}{26/3 -2}  ) =\frac{1}{2}
$. 
Similarly, 
     if $r= \hat r = 2$ and    
$q> p = 26$, then $ \hat p =   13$ 
and mean-square order $\hat \g = 
\frac{1}{2} \we \frac{1}{3} \we ( \frac{13}{3} - \frac{1}{13 -2}  ) =\frac{1}{3}
$.  
\end{remark}
\begin{remark}\label{rem11}
In \cite[Theorm 3.3]{Li_Mao2024spa}, 
the authors established a pointwise error estimate for the truncated EM scheme applied to SFDE \eqref{Eq0}, showing that  at a fixed time $T$, 
\begin{align}\label{tmp202601}
\E |X(T)- Y(T)|^{\hat p} \le C(T) \D^{\hat p/2}, \qu \hat p \in \Big [ 2, \frac{q}{1+ 3r/2} \Big  ).
\end{align}
Theorem \ref{Th943}, in contrast, provides an error bound between the segment processes
$X_T$ and $\bar Y_T$, thereby measuring the approximation quality over 
a whole time interval $[T-\tau, T]$ rather than at a single instant.
Since pointwise evaluation is a special case of segment evaluation, Theorem \ref{Th943} naturally extends the result of \cite[Theorem 3.3]{Li_Mao2024spa} from a pointwise to a pathwise setting.
It can therefore be regarded as a generalization and refinement of the earlier pointwise convergence theorem.
%
%
\end{remark}
\section{A numerical example}
In this section,  a numerical example is presented to verify the strong convergence rate of  segment sequence for truncated EM.
\begin{example}\label{example4072}
Consider the functional version of  the stochastic volatility model  \eqref{example407}.  Applying the \Holder inequality 
and  the elementary equality
\begin{align}\label{eq4007}
& -2(a^2 + b^2 + ab) |a-b|^2 = -2(a^2 + b^2) |a-b|^2  -2 ab |a-b|^2 \no \\
& \le -2(a^2 + b^2) |a-b|^2+ (a^2 + b^2) |a-b|^2 = - |a-b|^2(a^2 + b^2), \qu \f a, b \in \R,  
\end{align}
we have 
\begin{align*}
&2\langle \psi(0) - \bar{\psi}(0), f(\psi) - f(\bar{\psi}) \rangle + 
(q-1) |g(\psi) - g(\bar{\psi})|^2 \\
&= 2 \lan \psi(0) - \bar{\psi}(0), a_1 (\psi(0) - \bar{\psi}(0))-a_2 (\psi^3(0) - \bar{\psi}^3(0))\ran
+ (q-1) \Big | \int_{-1}^0 [ \psi^2 (\th) - \bar \psi^2 (\th) ]d \th  \Big |^2 \\
& \le  2 a_1|\psi(0) - \bar{\psi}(0)|^2 - 2a_2 |\psi(0) - \bar{\psi}(0)|^2
 \Big ( \psi^2(0) + \psi(0) \bar \psi(0) + \bar \psi^2(0)  \Big ) \\
&\qu + (q-1)
 \int_{-1}^0  |\psi(\th) - \bar{\psi}(\th)|^2 |\psi(\th) + \bar{\psi}(\th)|^2  d \th \\
 &\le  2 a_1|\psi(0) - \bar{\psi}(0)|^2  -  a_2 |\psi(0) - \bar{\psi}(0)|^2
\Big ( \psi^2(0)  + \bar \psi^2(0)  \Big )  
 + 2(q-1)
 \int_{-1}^0  |\psi(\th) - \bar{\psi}(\th)|^2 |\psi(\th) + \bar{\psi}(\th)|^2  d \th .
\end{align*}
If $a_2 > 2(q-1)$ and $q=27$, then Assumptions  \ref{AS2_Kha} and \ref{AS5_POLYNO} hold 
with $\a_0=2a_1$, $\a_1= a_2$, $\a_2 = 2 (q-1)$, $ \hat r = r =2$. 

To demonstrate the strong convergence rate of the truncated EM method, the numerical solution obtained with $\D = 2^{-12}$ is taken as the reference. 
For SFDE \eqref{example407} with parameters $a_0=3$, $a_1=10$, $a_2=53$, 
Fig. \ref{fig1} shows the graph of $\log (\sqrt{\E \| X_T - Y_T\| ^2})$ versus $\log (\D)$ for five step sizes $\D = [2^{-7}, 2^{-8}, 2^{-9}, 2^{-10}, 2^{-11}]$ at $T=10$. Here, $X_T$ and $Y_T$ represent the exact and numerical segments at the terminal time $T$, respectively. The expectation $\E$ is approximated by taking the average over 1000 sample realizations.
It can be seen that the observed convergence order of $0.5234$ is higher than the theoretical order of $1/6$ given in Remark \ref{rem1}. This indicates that our theoretical results are somewhat conservative.  
Therefore, bridging the gap between these theoretical bounds and the high convergence orders observed in practice remains an important goal for future work.

\end{example}
\begin{figure}[!t]
  \centering
  \includegraphics[width=8cm,height=6cm]{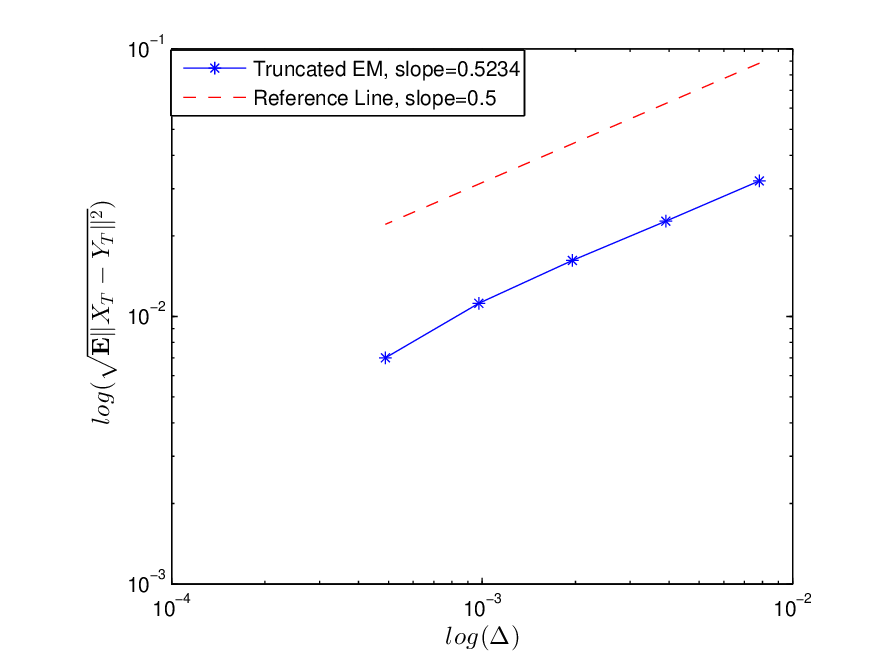}
\caption{Rate of strong convergence for Example \ref{example4072}  }
\label{fig1}
\end{figure}

\section*{Acknowledgments}
This work was supported  by the Natural Science Foundation of China (Nos. 12271003, 12501626, 62273003), the China Postdoctoral Science Foundation (No. 2023M733928), and
the
Natural Science Research Project of Anhui Educational Committee (No. 2023AH010011).
\bibliographystyle{model1-num-names}
\bibliography{refs_total202506}

\end{document}